\documentclass[12pt]{amsart}

\usepackage{amssymb,mathrsfs,amsmath,amsthm,color,bm}

\usepackage{enumerate}
\usepackage{graphicx}
\usepackage[centering]{geometry}
\usepackage{cite}
\theoremstyle{plain}
\newtheorem{thm}{Theorem}[section]
\newtheorem{definition}{Definition}[section]

\newtheorem{proposition}[thm]{Proposition}
\newtheorem{corollary}[thm]{Corollary}
\newtheorem{lemma}[thm]{Lemma}

\theoremstyle{remark}
\newtheorem{remark}{Remark}
\numberwithin{equation}{section}

\newcommand{\N}{\mathbb N}

\newcommand{\bz}{\mathbf{z}}

\begin{document}
	\title[Quantitative twisted recurrence properties]{Quantitative twisted recurrence properties for piecewise expanding maps on $[0,1]^d$} %%%%%%%%%%%%
	
		\author[Jiachang Li]{Jiachang Li}
		\address{Faculty of Innovation Engineering, Macau University of Science and Technology, Macau, {\rm	999078}, P. R. China}
		\email{3220001944@student.must.edu.mo }
	
		\author[Chao Ma]{Chao Ma}
		\address{Faculty of Innovation Engineering, Macau University of Science and Technology, Macau, {\rm	999078}, P. R. China}
		\email{cma@must.edu.mo }
	
	\date{\today}
	

	\let\thefootnote\relax
	
	\begin{abstract}
		Let $T:[0,1]^d \rightarrow[0,1]^d$ be a piecewise expanding map with an absolutely continuous (with respect to the $d$-dimensional Lebesgue measure $m_d$) $T$-invariant probability measure $\mu$. Let $\left\{\mathbf{r}_n\right\}$ be a sequence of vectors satisfying the conditons that $\mathbf{r}_n=\left(r_{n, 1}, \ldots, r_{n, d}\right) \in\left(\mathbb{R}_{\geq 0}\right)^d$, the sequence $\left\{\frac{\max _{1 \leq i \leq d}r_{n, i}}{\min _{1 \leq i \leq d}r_{n, i}}\right\}$ is bounded and  
		$\lim _{n \rightarrow \infty} \max _{1 \leq i \leq d}r_{n, i}=0$. Let $\left\{\delta_n\right\}$ be a sequence of non-negative real numbers with $
		\lim _{n \rightarrow \infty} \delta_n=0
		$. Under the assumptions that $\mu$ is exponentially mixing and its density is sufficiently regular, we prove that the $\mu$-measure of the following sets
		$$
		\mathcal{R}^f\left(\left\{\mathbf{r}_n\right\}\right)=\left\{\mathbf{x} \in[0,1]^d: T^n \mathbf{x} \in R\left(f(\mathbf{x}), \mathbf{r}_n\right) \text { for infinitely many } n \in \mathbb{N} \right\} 
		$$
		and
		$$
		\mathcal{R}^{f \times}\left(\left\{\delta_n\right\}\right)=\left\{\mathbf{x} \in[0,1]^d: T^n \mathbf{x} \in H\left(f(\mathbf{x}), \delta_n\right) \text { for infinitely many } n \in \mathbb{N} \right\}
		$$
		obeys zero-full laws determined by the convergence or divergence of natural volume sums. Here, $R(f(\mathbf{x}), \mathbf{r}_n)$ and $H(f(\mathbf{x}), \delta_n)$ represent ‘targets’ as, respectively, coordinate-parallel hyperrectangles with bounded aspect ratio, and hyperboloids, both centered at $f(\mathbf{x})$. $f: [0,1]^d \rightarrow [0,1]^d$ is a piecewise Lipschitz vector function. Our results not only unify quantitative recurrence properties and the shrinking target problem for piecewise expanding maps on $[0,1]^d$, but also reveal that the two problems and cross-component recurrence (one component of a point $\mathbf{x} \in \mathbb{R}^d$ in a measurable dynamical system returns close to another component of the point under a measure-preserving map $T$) can coexist in distinct directions on $[0,1]^d$. We also show that, under a additional mild condition on $f$, the bounded aspect ratio assumption can be removed while maintaining applicability to the two classical problems.
	\end{abstract} %%%%%%%%% 	\noindent \lipsum[1] \cite{1}
	\maketitle

	\section{Introduction}
	Let $(X, \mathcal{B}, \mu, T)$ be a measure-preserving dynamical system with a compatible metric $\rho$, that is, $(X, \rho)$ is a metric space, $\mathcal{B}$ is a Borel $\sigma$-algebra of $X$, and $\mu$ is an $T$-invariant probability measure. Suppose that $(X, \rho)$ is complete and separable, then Poincaré's recurrence theorem, one of the most foundational results within dynamical systems and ergodic theory (see \cite[Theorem 2.11]{1}), asserts that $\mu$-almost every point in a measurable dynamical system returns close to itself under a measure-preserving map $T$, i.e.
	$$
	\liminf _{n \rightarrow \infty} \rho\left(T^n x, x\right)=0.
	$$
	
	Boshernitzan proved the ﬁrst general quantitative recurrence results in \cite[Theorem 1.2]{2}. Specifically, assume that the $\alpha$-dimensional Hausdorff measure of $X$ is $\sigma$-finite for some $\alpha>0$. Then,
	$$
	\liminf _{n \rightarrow \infty} n^{1 / \alpha} \rho\left(T^n x, x\right)<\infty \quad \text { for  } \mu-a.e. x \in X\text {. }
	$$
	Here, $\mu-a.e.$ denotes $almost$ $every$ with respect to $\mu$. This result leads us to focus on the size of the following set:
	$$
	R(\psi):=\left\{x \in X: \rho\left(T^n x, x\right)<\psi(n) \text { for i.m. } n \in \mathbb{N}\right\}
	$$
	\noindent where $\psi: \mathbb{N} \rightarrow \mathbb{R}^{+}$ is a positive function and $i.m.$ denotes $infinitely$ $many$.
	
	The zero-full law for the $\mu$-measure of the recurrence set $R(\psi)$, often referred to as quantitative recurrence properties or dynamical Borel-Cantelli lemma for recurrence theory, has been extensively studied in various dynamical systems such as  \cite{3,4,5,6,7,8,9}; for similar results, see also \cite{10,11,12,13}. Particularly noteworthy is the recent work of Y.-B. He and L.-M. Liao [7], which proves quantitative recurrence properties for piecewise expanding maps on $[0,1]^d$.
	
	The shrinking target problem involves determining the rate at which the orbit of a $\mu$-typical point accumulates near a fixed point $y \in X$. To be more specific, for  $\psi: \mathbb{N} \rightarrow \mathbb{R}^{+}$, define 
	$$
	R^y(\psi):=\left\{x \in X: \rho\left(T^n x, y\right)<\psi(n) \text { for i.m. } n \in \mathbb{N}\right\}\text {. }
	$$
	Readers seeking a deeper understanding of the $\mu$-measure or similar results of such set are referred to \cite{14,15,16} and the references cited therein. Of particular relevance to our work is the recent study by B. Li, L.-M. Liao, S. Velani and E. Zorin \cite{16}, which investigates the shrinking target problem for matrix transformations of tori.
	
	These two problems are the most commonly studied, but they are treated separately and proved differently. A more general definition, known as twisted recurrence or modified shrinking target problem, unifies the shrinking target problem and recurrence. Specifically, for  $\psi: \mathbb{N} \rightarrow \mathbb{R}^{+}$ and a Borel measurable function $f: X \rightarrow X$, define 
	\begin{equation} \label{eq:f(x)}
		R^f(\psi):=\left\{x \in X: \rho\left(T^n x, f(x)\right)<\psi(n) \text { for i.m. } n \in \mathbb{N}\right\} .
	\end{equation}
	Clearly, if $f=Id_X$, then $ R^f(\psi)=R(\psi)$. In addition, if $f(x) \equiv y$, then $ R^f(\psi)=R^y(\psi)$. 
	
	The zero-full law for the $\mu$-measure of the set \eqref{eq:f(x)} has been the subject of significant study in various dynamical systems. For example, D. Kleinbock and J. Zheng \cite{17} established the $\mu$-measure results for \eqref{eq:f(x)} in conformal dynamical systems with pseudo-Markov partitions, including prominent examples such as the Gauss map and the multiplication map $M_b$ for integers $b \geq 2$. For $M_\beta$, where $\beta>1$, F. Lü, B.-W. Wang and J. Wu \cite{18} established the zero-full criteria with a Lipschitz function $f:[0,1] \rightarrow[0,1]$. For dynamical systems with exponential decay of correlations, see J. Zheng \cite{19}. Furthermore, Z.-P. Shen \cite{20} presented both the $\mu$-measure and Hausdorff dimension results on self-conformal sets generated by a conformal iterated function system satisfying the open set condition. Y.-F. Wu \cite{21} also obtained the Lebesgue measure and Hausdorﬀ
	dimension results for $X=[0,1]^2$ in \eqref{eq:f(x)} with Lipschitz functions $f:[0,1]^2 \rightarrow[0,1]$. Moreover, a number of related Hausdorff dimension results can be found in \cite{22,23,24,25}. Notably, when $T$ is an expanding real matrix transformation of tori, N. Yuan and S.-L. Wang \cite{22} obtained the  Hausdorff dimension results of the $d$-dimensional ‘hyperrectangle’ form of \eqref{eq:f(x)}  for any Lipschitz coordinate-wise vector function $f:[0,1)^d \rightarrow [0,1)^d$; that is the Lipschitz vector function $f(\mathbf{x})=\left(f_{1^{'}}\left(x_1\right), f_{2^{'}}\left(x_2\right), \cdots, f_{d^{'}}\left(x_d\right)\right)$, where $f_{i^{'}}: [0,1) \rightarrow [0,1)$ denotes the $i$-th component function of $f$ for $1 \leq i^{'} \leq d$.
	
	As discussed above, extensive research has investigated twisted recurrence or modified shrinking target problem, yet cross-component recurrence of $X=\mathbb{R}^d$ remains unexplored. The cross-component recurrence refers to the asymptotic behaviour of dynamical system orbits that one component of a point $\mathbf{x} \in \mathbb{R}^d$ in a measurable dynamical system returns close to another component of the ponit under a measure-preserving map $T$. Specifically, for  $\psi: \mathbb{N} \rightarrow \mathbb{R}^{+}$ and components $x_i$ and $x_j$ of a point $\mathbf{x} \in \mathbb{R}^d$ with $i \neq j$,  define
	$$
	R^{i,j}(\psi):=\left\{\mathbf{x} \in \mathbb{R}^d: \rho\left(T^n x_i, x_j\right)<\psi(n) \text { for i.m. } n \in \mathbb{N}\right\} .
	$$
	By defining the function $f$ in \eqref{eq:f(x)} as a vector function $f:\mathbb{R}^d \rightarrow \mathbb{R}^d$; that is $$f(\mathbf{x})=\left(f_1\left(x_1, \ldots, x_d\right), f_2\left(x_1, \ldots, x_d\right), \cdots, f_d\left(x_1, \ldots, x_d\right)\right),$$ 
	where $f_i: \mathbb{R}^d \rightarrow \mathbb{R}$ denotes the $i$-th component function of $f$ for $1 \leq i \leq d$, we can introduce the cross-component recurrence into the framework of twisted recurrence. We also notice that if $f_i\left(x_1, \ldots, x_d\right)=x_j$, then $R^{i,j}(\psi) \subset R^f(\psi)$. This phenomenon manifests through the asymptotic behaviour between components, revealing nontrivial interdimensional coupling mechanisms in dynamical systems.
	
	The primary objective of this paper is to establish quantitative twisted recurrence properties in the non-conformal case, with a focus on piecewise expanding maps (including expanding real matrix transformations) and the ‘targets’ are either hyperrectangles or hyperboloids. Moreover, by adopting the framework and methods of \cite{7} and \cite{16}, we can unify some results from these two studies.
	
	Throughout the paper, we take $X$ to be the unit hypercube $[0,1]^d$ endowed with the maximum norm $|\cdot|$, i.e. for any $\mathbf{x}=\left(x_1, \ldots, x_d\right) \in[0,1]^d$, $|\mathbf{x}|=\max \left\{\left|x_1\right|, \ldots,\left|x_d\right|\right\}$. Let $|\cdot|_{\text {min }}$ denote minimum norm such that $|\mathbf{x}|_{\text {min }}=\min \left\{\left|x_1\right|, \ldots,\left|x_d\right|\right\}$. Let $m_d$ and $m_{d-1}$  denote the Lebesgue measures of dimensions $d$ and $d-1$, respectively. Given a set $A \subset \mathbb{R}^d$, let $\bar{A}, \partial A, A(\varepsilon),|A|$ and $\# A$ denote the closure, boundary, $\varepsilon$-neighbourhood, diameter, and cardinal number of $A$, respectively. 
	Following the classical text of Federer \cite[Section 3.2.37]{26}, the $(d-1)$-dimensional upper and lower Minkowski content of $A$ are defined, respectively as
	$$
	M^{*(d-1)}(A):=\underset{\epsilon \rightarrow 0^{+}}{\lim \sup } \frac{m_d(A(\epsilon))}{\epsilon} \quad \text { and } \quad M_*^{(d-1)}(A):=\liminf _{\epsilon \rightarrow 0^{+}} \frac{m_d(A(\epsilon))}{\epsilon} \text {. }
	$$
	If these upper and lower Minkowski contents are equal, then their common value is called the $(d-1)$-dimensional Minkowski content of $A$, denoted by $M^{(d-1)}(A)$. Specifically, Federer \cite[Theorem 3.2.39]{26} shows that if $A$ is a closed $(d-1)$-rectifiable subset of $\mathbb{R}^d$ (i.e. the image of a bounded set from $\mathbb{R}^{(d-1)}$ under a Lipschitz function), then the $M^{(d-1)}(A)$ exists and is equal to the $(d-1)$-dimensional Hausdorff measure of $A$, which is also a constant $c_M >0$ multiple of $m_{d-1}(A)$.
	
	Let us now introduce the class of measure-preserving systems $\left([0,1]^d, T, \mu\right)$, as defined in \cite[Definition 1.1]{7}.
	
	\begin{definition}[{\cite[Definition 1.1 (Piecewise expanding map)]{7}}] \label{pem}  We say that $T:[0,1]^d \rightarrow[0,1]^d$ is a piecewise expanding map if there is a finite family $\left(U_i\right)_{i=1}^Q$ of non-empty, pairwise disjoint, and connected open subsets in $[0,1]^d$ with $ \bigcup_{i=1}^Q\overline{U_i}=[0,1]^d$ such that the following statements hold.
		\begin{enumerate}
			\item[(i)] The map $T$ is injective and can be extended to a $C^1$ map on each $\bar{U}_i$. Moreover, there exists a constant $L>1$ such that
			$$
			\inf _i \inf _{\mathbf{x} \in U_i}\left\|D_{\mathbf{x}} T\right\| \geq L,
			$$
			where $D_{\mathbf{x}} T$ is the derivative of $T$ at $\mathbf{x}$ determined by
			$$
			\lim _{|\mathbf{z}| \rightarrow 0} \frac{\left|T(\mathbf{x}+\mathbf{z})-T(\mathbf{x})-\left(D_{\mathbf{x}} T\right)(\mathbf{z})\right|_2}{|\mathbf{z}|_2}=0
			$$
			and $||D_{\mathbf{x}} T \|:=\sup _{\mathbf{z} \in \mathbb{R}^d}\left|\left(D_{\mathbf{x}} T\right)(\mathbf{z})\right| /|\mathbf{z}|$. Here $ |\bz|_2=\sqrt{z_1^2+\cdots+z_d^2} $. 
			\item[(ii)] The boundary of $U_i$ is included in a $C^1$ piecewise embedded compact submanifold of codimension one. In particular, there exists a constant $K$ such that
			$$
			\max _{1 \leq i \leq Q} M^{*(d-1)}\left(\partial U_i\right) \leq K
			$$
		\end{enumerate}
	\end{definition}

	\begin{remark}
		Let $T$ be a real matrix transformation and its corresponding matrix $\mathbf{T}$ be a $d \times d$ matrix with real coefficients. Then, $T$ determines a self-map of $[0,1)^d$; namely, it sends $\mathbf{x} \in[0,1)^d$ to $\mathbf{x} \mathbf{T} \bmod 1$. 
		
		We say that if $\mathbf{T}$ is a $d \times d$ real, non-singular matrix with the modulus of all eigenvalues strictly larger than 1, then $T$ is a piecewise expanding map. The reasons are briefly as follows. First, $T$ is injective  because $\mathbf{T}$ is  non-singular. And since $\mathbf{T}$ is a $d \times d$ real matrix, then $\mathbf{x} \mathbf{T}$ is linear, making $T$ a $C^1$ map on $[0,1]^d$. Moreover, $T$ expands in every direction because the modulus of all eigenvalues of $\mathbf{T}$ is strictly larger than $1$. Finally, for each $1 \leq i \leq Q$, $U_i$ is a $d$-dimensional open subset in $[0,1)^d$, so there exists a constant $K$ such that	$
		\max _{1 \leq i \leq Q} M^{*(d-1)}\left(\partial U_i\right) \leq K
		$. For example, when $d=2$, for $\mathbf{T}=\left(\begin{array}{cc}\frac{3}{2} & \sqrt{2} \\ 1 & -2\end{array}\right)$, the finite family $\left(U_i\right)_{i=1}^{10}$ of non-empty, pairwise disjoint, and connected open subsets in $[0,1)^2$ with $ \bigcup_{i=1}^{10}\overline{U_i}=[0,1]^2 $ is shown in Figure \ref{f1}.
		
		\begin{figure}[ht]\label{f1}
			\centering
			\includegraphics[scale=0.5]{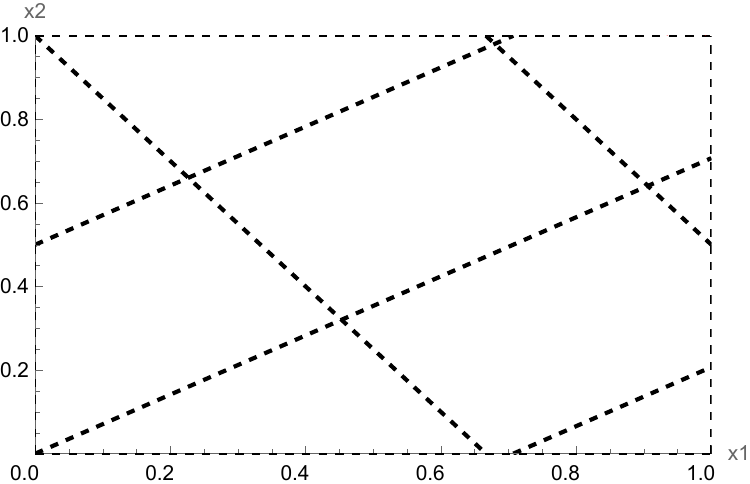}
			\caption{The finite family $\left(U_i\right)_{i=1}^{10}$ of non-empty, pairwise disjoint, and connected open subsets in $[0,1)^2$ for $T \mathbf{x}=(\frac{3}{2}x_1+x_2 \bmod 1,\sqrt{2}x_1-2x_2 \bmod 1)$}
		\end{figure}
	\end{remark}
%%The equation corresponding to the dashed line in the figure is: $\frac{3}{2} x_1+x_2=1$, $\frac{3}{2} x_1+x_2=2$, $\sqrt{2} x_1-2 x_2=-1$, $\sqrt{2} x_1-2 x_2=0$, and $\sqrt{2} x_1-2 x_2=1$. 
	
	\begin{definition}[Cylinders of order $n$] The collection of cylinders of order $n$ is defined as 
		$$
		\mathcal{F}_n:=\left\{U_{i_0} \cap T^{-1} U_{i_1} \cap \cdots \cap T^{-(n-1)} U_{i_{n-1}}: 1 \leq i_0, i_1, \ldots, i_{n-1} \leq Q\right\}.
		$$
	\end{definition}
	
	\begin{proposition}[{\cite[Proposition 1]{16}}]\label{l:Li}
		Let $T$ be a real, non-singular matrix transformation of the torus $\mathbb{T}^d:=\mathbb{R}^d / \mathbb{Z}^d$ and $\mathbf{T}$ be a $d \times d$ non-singular matrix with real coefficients. Suppose that all eigenvalues of $\mathbf{T}$ are of modulus strictly larger than 1. Then
		\begin{enumerate}
			\item[(i)] there exists an absolutely continuous (with respect to $m_d$) invariant probability measure (acim) $\mu$,
			\item[(ii)] the support $A \subseteq \mathbb{T}^d$ of any acim $\mu$ can be decomposed into finitely many disjoint measurable sets $A_1, \ldots, A_s$ such that for each $1 \leq i \leq s$ the restriction $\left.\mu\right|_{A_i}$ of $\mu$ to $A_i$ is ergodic and is equivalent to the restriction $\left.m_d\right|_{A_i}$ of $m_d$ to $A_i$, 
			\item[(iii)] each ergodic component $A_i$ in (ii) can in turn be decomposed into finitely many disjoint measurable sets $A_{i 1}, \ldots, A_{i p_i}$ such that for each $1 \leq j \leq p_i$ the restriction $\left.\mu\right|_{A_{i j}}$ is mixing with respect to $T^{p_i}$, 
			\item[(iv)] on each mixing component $A_{i j}$ in (iii), the restriction $\left.\mu\right|_{A_{i j}}$ is exponentially mixing with respect to $\left(T^{p_i}, \mathcal{C}\right)$ for any collection $\mathcal{C}$ of subsets $E$ of $A_{i j}$ satisfying the bounded property
			$$
			(\mathbf{B}): \sup _{E \in \mathcal{C}} M^{*(d-1)}(\partial E)<\infty.
			$$
		\end{enumerate}
	\end{proposition}

	According to Proposition \ref{l:Li}, we can assume that  $\mu$ is an absolutely continuous (with respect to $m_d$) $T$-invariant probability measure. Under this assumption, we have
	\[\mu\bigg(\bigcup_{1\le i\le Q}U_i \bigg)=1.\]
	Futhermore, Proposition \ref{l:Li} characterizes the collection $\mathcal{C}$ of ‘target’ sets as subsets $E$ of $\mathbb{T}^d$ for which the boundary $\partial E$ has bounded $(d-1)$-dimensional upper Minkowski content; that is, $\sup _{E \in \mathcal{C}} M^{*(d-1)}(\partial E)<\infty$. Following the conventions of \cite{16}, we employ the definition below throughout this work.
	
	\begin{definition} [$\phi(n)$-mixing] \label{mixing}
		For any $F, G \in \mathcal{C}$, we have
		\begin{equation}
			\left|\mu\left(F \cap T^{-n} G\right)-\mu(F) \mu(G)\right| \leq  \mu(G) \phi(n),
			\label{Eq.1.2}
		\end{equation}
		where $\phi: \mathbb{N} \rightarrow(0,1)$ is a positive function. The $\phi(n)$-mixing has three progressively enhanced forms.
		\begin{enumerate}
			\item[(i)] $\mu$ is $\Sigma$-mixing (short for summable-mixing) with respect to $(T, \mathcal{C})$, if the series $\sum_{n=1}^{\infty} \phi(n)$ converges. 
			\item[(ii)]  $\mu$ is polynomial-mixing with respect to $(T, \mathcal{C})$,  if $\lim _{n \rightarrow \infty} \phi(n) \cdot n^\alpha=0$, $\forall \alpha>0.$  
			\item[(iii)] $\mu$ is exponential-mixing with respect to $(T, \mathcal{C})$,  if there exist constants $c>0$ and $\tau>0$ such that 
			$\phi(n)=c e^{-\tau n}.$
		\end{enumerate}
	\end{definition}
	
	We introduce these three forms of mixing to ensure precision, as each lemma involving mixing necessitates invoking the weakest mixing assumptions compatible with its proof. For example, the mixing assumptions diverge between the convergence part and divergence part of our results.

	We study two special sets that are associated with the classical Diophantine approximation theory. For any $\mathbf{x} \in[0,1]^d$, $\mathbf{r} \in\left(\mathbb{R}_{\geq 0}\right)^d$, and $\delta>0$, and for a Borel measurable vector function $f: [0,1]^d \rightarrow [0,1]^d$, let
	$$
	R(f(\mathbf{x}), \mathbf{r}):=\prod_{i=1}^d B\left(f_i\left(x_{1}, x_{2}, \ldots, x_{d}\right), r_i\right)
	$$
	and
	$$
	H(f(\mathbf{x}), \delta):=f(\mathbf{x})+\left\{\mathbf{z} \in[-1,1]^d:\left|z_1 \cdots z_d\right|<\delta\right\} \text{,}
	$$
	where $R(f(\mathbf{x}), \mathbf{r})$ is a coordinate-parallel hyperrectangle and $H(f(\mathbf{x}), \delta)$ is a hyperboloid, both centered at  $f(\mathbf{x})$. And if $\frac{\left|\mathbf{r}\right|}{\left|\mathbf{r}\right|_{\text {min }}}$ is bounded, then $R\left(f(\mathbf{x}), \mathbf{r}\right)$ represents a coordinate-parallel hyperrectangle with bounded aspect ratio.
	Let $\left\{\mathbf{r}_n\right\}$ be a sequence of vectors with $\mathbf{r}_n=\left(r_{n, 1}, \ldots, r_{n, d}\right) \in\left(\mathbb{R}_{\geq 0}\right)^d$, define
	$$
	\mathcal{R}^f\left(\left\{\mathbf{r}_n\right\}\right):=\left\{\mathbf{x} \in[0,1]^d: T^n \mathbf{x} \in R\left(f(\mathbf{x}), \mathbf{r}_n\right) \text { for i.m. } n  \in \mathbb{N}\right\} \text {. }
	$$
	If all components of the vector $\mathbf{r}_n$ are equal, that is $r_{n, 1}=\cdots=r_{n, d}=\psi(n)$, then the aspect ratio $\frac{\left|\mathbf{r}_n\right|}{\left|\mathbf{r}_n\right|_{\text {min }}} \equiv 1$, and $\mathcal{R}^f\left(\left\{\mathbf{r}_n\right\}\right)$ coincides with $R^f(\psi)$ defined in \eqref{eq:f(x)} for $X=[0,1]^d$.
	
	Next, let $\left\{\delta_n\right\}$ be a sequence of non-negative real numbers, define
	$$
	\mathcal{R}^{f \times}\left(\left\{\delta_n\right\}\right):=\left\{\mathbf{x} \in[0,1]^d: T^n \mathbf{x} \in H\left(f(\mathbf{x}), \delta_n\right) \text { for i.m. } n \in \mathbb{N}\right\}\text {. }
	$$
	It is worth noting that we do not assume the sequences $\left\{\mathbf{r}_n\right\}$ and $\left\{\delta_n\right\}$ to be non-increasing, as in \cite{7}.
	
	Let us now specify the class of functions $f$ that can be treated by our technique.

	\begin{definition} [Piecewise Lipschitz vector function]\label{pf}  Let $p>0$. Say that $f:[0,1]^d \rightarrow[0,1]^d$ is Lipschitz if
		$$
		|f(\mathbf{x})-f(\mathbf{y})| \leq p|\mathbf{x}-\mathbf{y}|, \quad \forall \mathbf{x}, \mathbf{y} \in[0,1]^d \text {, } \mathbf{x} \neq \mathbf{y}.
		$$
		Furthermore, $f$ is piecewise Lipschitz if there exist finite measurable subsets $U_j$ in $[0,1]^d$ and $\left.f\right|_{U_j}: [0,1]^d \rightarrow [0,1]^d$ is a Lipschitz function, $j \in \mathcal{I}$, such that $\mu\left([0,1]^d \backslash \bigcup_{j}U_j \right)=0$.
	\end{definition}
	For example, when $d=1$, the function $f(x)=2x \bmod 1$ is piecewise Lipschitz but not Lipschitz on $[0,1)$.
	
	To avoid modifying the existing conditions in \cite{7}, we need to use a property that holds for almost all points in $[0,1]^d$ and we demand that the property remains  satisfied by almost all points and their images under $f$. To ensure this invariance, we introduce an additional technical assumption on $f$.
	
	\begin{definition} [$\mu \circ f^{-1} \ll \mu$] \label{ab}
		We say that $\mu \circ f^{-1}$ is absolutely continuous with respect to $\mu$ (written $\mu \circ f^{-1} \ll \mu$) if for any measurable subset $A \subset [0,1]^d$ with $\mu(A)=0$, $\mu \circ f^{-1} (A)=0$.
	\end{definition}

	\begin{remark} Throughout this paper, the notation $f^{-1}$ denotes the preimage operator under $f$, not the inverse function. Specifically, for any set $A \subset [0,1]^d$,  $f^{-1}(A)=\left\{\mathbf{x} \in[0,1]^d: f(\mathbf{x}) \in A\right\}$.
	\end{remark}
	
	\begin{remark} While the condition of Definition \ref{ab} on $f$ is not essential for establishing quantitative twisted recurrence properties for matrix transformations of tori, it becomes indispensable when relaxing the regularity assumptions on the density of $\mu$ as in \cite{7}.
		In fact, many functions satisfy $m_d \circ f^{-1} \ll m_d$ and we will introduce such functions in Section \ref{abp}.
	\end{remark}
	We now turn to our main theorems.
	
	\begin{thm} \label{12}
		Let $T:[0,1]^d \rightarrow[0,1]^d$ be a piecewise expanding map. Suppose that the absolutely continuous (with respect to $m_d$) $T$-invariant probability measure $\mu$ is exponentially mixing. Let $f: [0,1]^d \rightarrow[0,1]^d$ be a piecewise  Lipschitz vector function. Assume that one of the following conditions holds.
		\begin{enumerate}
			\item The density $h$ of $\mu$ belongs to $L^q\left(m_d\right)$ for some $q>1$. And $m_d \circ f^{-1} \ll m_d$.  \label{121}
			\item The density $h(\mathbf{x})$ of $\mu$ is bounded from above by $\mathfrak{c} > 0$ for all $\mathbf{x} \in[0,1]^d$.  \label{122}
		\end{enumerate}
		For a sequence of vectors $\left\{\mathbf{r}_n\right\}$ in $\left(\mathbb{R}_{\geq 0}\right)^d $ satisfying the conditons that the sequence $\left\{\frac{\left|\mathbf{r}_n\right|}{\left|\mathbf{r}_n\right|_{\text {min }}}\right\}$ is bounded and 
		$\lim _{n \rightarrow \infty}\left|\mathbf{r}_n\right|=0$, we have
		\begin{equation} \label{111}
			\mu\left(\mathcal{R}^f\left(\left\{\mathbf{r}_n\right\}\right)\right)= \begin{cases}0 & \text { if } \sum_{n=1}^{\infty} r_{n, 1} \cdots r_{n, d}<\infty, \\ 1 & \text { if } \sum_{n=1}^{\infty} r_{n, 1} \cdots r_{n, d}=\infty .\end{cases}
		\end{equation}
		Furthermore, if we further suppose that $f(\mathbf{x})=\left(f_{1^{'}}\left(x_1\right), f_{2^{'}}\left(x_2\right), \cdots, f_{d^{'}}\left(x_d\right)\right)$, for a sequence of vectors $\left\{\mathbf{r}_n\right\}$ in $\left(\mathbb{R}_{\geq 0}\right)^d $ with
		$\lim _{n \rightarrow \infty}\left|\mathbf{r}_n\right|=0$,  the conclusion \eqref{111} holds.
	\end{thm}
	
	\begin{thm} \label{13}
		Let $T:[0,1]^d \rightarrow[0,1]^d$ be a piecewise expanding map. Suppose that the absolutely continuous (with respect to $m_d$) $T$-invariant probability measure $\mu$ is exponentially mixing. Let $f: [0,1]^d \rightarrow[0,1]^d$ be a piecewise  Lipschitz vector function.
		Assume that one of the following conditions holds.
		\begin{enumerate}
			\item  There exists an open set $V$ with $\mu(V)=1$ such that the density $h$ of $\mu$, when restricted to $V$, is bounded from above by $\mathfrak{c} \geq 1$ and from below by $\mathfrak{c}^{-1}$. And $\mu \circ f^{-1} \ll \mu$. \label{131}
			\item The density $h(\mathbf{x})$ of $\mu$ is bounded from above by $\mathfrak{c} \geq 1$ and bounded from below by $\mathfrak{c}^{-1}$ for all $\mathbf{x} \in[0,1]^d$. \label{132}
		\end{enumerate}
		For  a sequence of non-negative real numbers $\left\{\delta_n\right\}$ with $
		\lim _{n \rightarrow \infty} \delta_n=0
		$, we have 
		$$
		\mu\left(\mathcal{R}^{f\times}\left(\left\{\delta_n\right\}\right)\right)= \begin{cases}0 & \text { if } \sum_{n=1}^{\infty} \delta_n\left(-\log \delta_n\right)^{d-1}<\infty, \\ 1 & \text { if } \sum_{n=1}^{\infty} \delta_n\left(-\log \delta_n\right)^{d-1}=\infty .\end{cases}
		$$
	\end{thm}
	
	\begin{remark}
		When the targets are coordinate-parallel hyperrectangles, the Zygmund differentiation theorem (see Theorem \ref{ZG}) applies. Consequently, we adopt a weaker density assumption (Condition \ref{121} of Theorem \ref{12}), which aligns with the regularity theorem in \cite{7}. However, the assumption necessitates the constraint $m_d \circ f^{-1} \ll m_d$ on $f$, which also can unify recurrence, the shrinking target problem for $m_d-a.e. y \in [0,1]^d$, and the cross-component recurrence and reveals that the three problems can coexist in distinct directions on $[0,1]^d$ for some cases, see details in Section \ref{abp}. When the density assumption is strengthened to Condition \ref{122} of Theorem \ref{12} (particularly suited to matrix transformations of tori), the constraint $m_d \circ f^{-1} \ll m_d$ on $f$ can be removed.
	\end{remark}
	
	\begin{remark}
		Theorem \ref{12} mainly applies to targets that are coordinate-parallel hyperrectangles with bounded aspect ratio (including balls). For piecewise Lipschitz coordinate-wise vector functions $f$, Theorem \ref{12} remains valid for general coordinate-parallel hyperrectangles. In this case, we can unify the shrinking target problem and recurrence, and show that the two problems can coexist in distinct directions on $[0,1]^d$, but cross-component recurrence is absent.
	\end{remark}
	
	\begin{remark}
		For the hyperboloid set, no direct analogue diﬀerentiation theorem is presently available for us. Under Condition \ref{131} of Theorem \ref{13} that aligns with the regularity assumption in \cite{7}, we can only unify recurrence and the shrinking target problem for $\mu-a.e. y \in [0,1]^d$ as best as we can.
	\end{remark}
	As direct consequences of Theorems \ref{12} and \ref{13}, we obtain the following corollaries.
	\begin{corollary}
		Let $\mathbf{T}$ be a $d \times d$ real, non-singular  matrix with the modulus of all eigenvalues strictly larger than 1. Let $f: [0,1)^d \rightarrow[0,1)^d$ be a piecewise  Lipschitz vector function. Suppose that $\mathbf{T}$ satisfies one of the following conditions. 
		\begin{enumerate}
			\item All eigenvalues of $\mathbf{T}$ are of modulus strictly larger than $1+\sqrt{d}$. 
			\item $\mathbf{T}$ is diagonal. 
			\item $\mathbf{T}$ is an integer matrix. 
		\end{enumerate}
		For a sequence of vectors $\left\{\mathbf{r}_n\right\}$ in $\left(\mathbb{R}_{\geq 0}\right)^d $ satisfying the conditons that the sequence $\left\{\frac{\left|\mathbf{r}_n\right|}{\left|\mathbf{r}_n\right|_{\text {min }}}\right\}$ is bounded and 
		$\lim _{n \rightarrow \infty}\left|\mathbf{r}_n\right|=0$, we have
		\begin{equation} \label{112}
			\mu\left(\mathcal{R}^f\left(\left\{\mathbf{r}_n\right\}\right)\right)= \begin{cases}0 & \text { if } \sum_{n=1}^{\infty} r_{n, 1} \cdots r_{n, d}<\infty, \\ 1 & \text { if } \sum_{n=1}^{\infty} r_{n, 1} \cdots r_{n, d}=\infty .\end{cases}
		\end{equation}
		Furthermore, if we further suppose that $f(\mathbf{x})=\left(f_{1^{'}}\left(x_1\right), f_{2^{'}}\left(x_2\right), \cdots, f_{d^{'}}\left(x_d\right)\right)$, for a sequence of vectors $\left\{\mathbf{r}_n\right\}$ in $\left(\mathbb{R}_{\geq 0}\right)^d $ with
		$\lim _{n \rightarrow \infty}\left|\mathbf{r}_n\right|=0$, 
		the conclusion \eqref{112} holds.
	\end{corollary}
	\begin{remark}
		The conditions stated in items (1)–(3) are consistent with those in \cite[Theorem 1.5]{7} and \cite[Theorem 3-5]{16}. According to Proposition \ref{l:Li} and Saussol's result \cite[Theorem 6.1]{27}, if $\mathbf{T}$ satisfies any of the three conditions above, then the absolutely continuous (with respect to $m_d$) $T$-invariant probability measure $\mu$ is exponentially mixing. Furthermore, Saussol \cite[Proposition 3.4 and Theorem 5.1 (ii)]{27} demonstrated that the density $h$ of $\mu$ is bounded from above. Hence, Theorem \ref{12} applies.
	\end{remark}
	
	\begin{remark}
		According to \cite[Proof of Lemma 3]{16}, if $\mathbf{T}$ satisfies the first conditon in the above corollary, then $\mu$ is the unique maximal entropy acim. Furthermore, according to \cite[Proof of Theorem 5]{16}, if $\mathbf{T}$ satisfies the third condition of the above corollary, then $h \equiv 1$, implying that $\mu$ is $m_d$.
	\end{remark}
	
	\begin{corollary}
		Let $\mathbf{T}$ be a $d \times d$ real, non-singular  matrix with the modulus of all eigenvalues strictly larger than 1. Let $f: [0,1)^d \rightarrow[0,1)^d$ be a piecewise  Lipschitz vector function. Suppose that $\mathbf{T}$ satisfies one of the following conditions.
		\begin{enumerate}
			\item $\mathbf{T}$ is diagonal with all eigenvalues belonging to $(-\infty,-(\sqrt{5}+1) / 2] \cup(1,+\infty)$. 
			\item $\mathbf{T}$ is an integer matrix. 
		\end{enumerate}
		For a sequence of non-negative real numbers $\left\{\delta_n\right\}$ with $
		\lim _{n \rightarrow \infty} \delta_n=0
		$,  we have 
		$$
		\mu\left(\mathcal{R}^{f \times}\left(\left\{\delta_n\right\}\right)\right)= \begin{cases}0 & \text { if } \sum_{n=1}^{\infty} \delta_n\left(-\log \delta_n\right)^{d-1}<\infty, \\ 1 & \text { if } \sum_{n=1}^{\infty} \delta_n\left(-\log \delta_n\right)^{d-1}=\infty.\end{cases}
		$$
	\end{corollary}
	
	\begin{remark}
		According to \cite[Proposition 2]{16}, if $\mathbf{T}$ satisfies the first condition of the above corollary, then the density $h(\mathbf{x})$ of $\mu$ is bounded from above by $\mathfrak{c} \geq 1$ and below by $\mathfrak{c}^{-1}$ for all $\mathbf{x} \in [0,1)^d$, where $\mu$ is the Parry-Yrrap measure. Thus, Theorem \ref{13} applies. 
	\end{remark}
	
	To enhance conciseness, we focus on proving  Theorem \ref{12} under Condition \ref{121}. Furthermore, when the assumption on the density $h$ is strengthened as in Condition \ref{122} of Theorem \ref{12}, we further demonstrate that the constraint $m_d \circ f^{-1} \ll m_d$ on  $f$ can be entirely eliminated. Additionally, throughout the proof, we also show that the bounded aspect ratio assumption can be removed if $f$ is a coordinate-wise vector function. And the proof of Theorem \ref{13} follows identical reasons. 
	
	The paper is organized as follows: Section \ref{abp} introduces a class of functions $f$ satisfying $m_d \circ f^{-1} \ll m_d$. Subsequently, Sections \ref{RC} and \ref{RD} shows the full proof of Theorem \ref{12}. Finally, Sections \ref{HC} and \ref{HD} establish the convergence and divergence part of Theorem \ref{13}, respectively. 
	
	\section{Further investigations on functions $f$ satisfying $\mu \circ f^{-1} \ll \mu$} \label{abp}
	
	\begin{lemma} \label{DJ} $\mu \circ f^{-1} \ll \mu$ if and only if, for $\Omega \subset[0,1]^d$ with $\mu(\Omega)=1$, $\mu \circ f^{-1}(\Omega)=1$.
	\end{lemma} 
	
	\begin{proof}  We begin by noting that 
		$$
		\begin{aligned}
			& \mu \circ f^{-1}(\Omega)=\int \chi_{f^{-1}(\Omega)}(\mathbf{x}) d \mu(\mathbf{x})=\int \chi_{\Omega}(f(\mathbf{x})) d \mu(\mathbf{x}) \\
			& =\int \chi_{[0.1]^d}(f(\mathbf{x}))-\chi_{\Omega^c}(f(\mathbf{x})) d \mu(\mathbf{x})=\int \chi_{[0,1]^d}(f(\mathbf{x})) d \mu(\mathbf{x})-\int \chi_{f^{-1}\left(\Omega^c\right)}(\mathbf{x}) d \mu(\mathbf{x}) \\
			& =1-\mu \circ f^{-1}\left(\Omega^c\right) .
		\end{aligned}
		$$
		Here, $\chi: [0,1]^d \rightarrow\{0,1\}$ is a characteristic function.  Clearly, the lemma holds and we conclude that if $\mu(\Omega) = 1$, then $\mu \circ f^{-1}(\Omega) = 1$;  that is, $f$ satisfying Definition \ref{ab} is also a full-$\mu$-measure-preserving vector function.
	\end{proof}
	
	\begin{thm}[{\cite[Theorem 2.29 (Zygmund differentiation theorem)]{28}}] \label{ZG}  Let $\left\{\mathbf{r}_n\right\}$ be a sequence of positive vectors with $\lim _{n \rightarrow \infty}\left|\mathbf{r}_n\right| \rightarrow 0$. If $h \in L^q\left(m_d\right)$ for some $q>1$, then
		$$
		\lim _{n \rightarrow \infty} \frac{\int_{R\left(\mathbf{x}, \mathbf{r}_n\right)} h(\mathbf{z}) \mathrm{d} m_d(\mathbf{z})}{m_d\left(R\left(\mathbf{x}, \mathbf{r}_n\right)\right)}=h(\mathbf{x}) \quad \text { for } m_d \text {-a.e. } \mathbf{x}\text {. } 
		$$
	\end{thm}
	\begin{remark} \label{CDXZ}
		Zygmund differentiation theorem applies specifically to coordinate-parallel hyperrectangles, and does not generally apply to other cases.
	\end{remark}
	
	Since $h$ belongs to $L^q\left(m_d\right)$ for some $q>1$ in Theorem \ref{12}, it follows from Theorem \ref{ZG} that 
	$$
	0 \leq \lim _{n \rightarrow \infty} \frac{\mu\left(R\left(\mathbf{x},  \mathbf{r}_n\right)\right)}{m_d\left(R\left(\mathbf{x}, \mathbf{r}_n\right)\right)}=\lim _{n \rightarrow \infty} \frac{\mu\left(R\left(\mathbf{x},  \mathbf{r}_n\right)\right)}{2^d r_{n, 1} \cdots r_{n, d}}=h(\mathbf{x})<\infty \quad \text { for } m_d \text {-a.e. } \mathbf{x} \in[0,1]^d\text {. }
	$$
	This implies that $\exists \Omega_1 \subset[0,1]^d$ with $\Omega_1=\bigcup_{k=1}^{\infty} \bigcup_{l=1}^{\infty} Z(k, l)$,  
	$$
	m_d \left(\bigcup_{k=1}^{\infty} \bigcup_{l=1}^{\infty} Z(k, l)\right)=1 \text {, }
	$$
	where
	$$
	Z(k, l):=\left\{\mathbf{x} \in[0,1]^d: \mu\left(R\left(\mathbf{x},  \mathbf{r}_n\right)\right) \leq k r_{n, 1} \cdots r_{n, d} \text { for all } n \geq l\right\} \text {. }
	$$
	Let $\mu$ be $m_d$ in Lemma \ref{DJ}, then $\exists \Omega_2 \subset[0,1]^d$, defined by $\Omega_2=f^{-1}(\bigcup_{k=1}^{\infty} \bigcup_{l=1}^{\infty} Z(k, l))=\bigcup_{k=1}^{\infty} \bigcup_{l=1}^{\infty} M(k, l)$, 
	$$ 
	m_d\left(\bigcup_{k=1}^{\infty} \bigcup_{l=1}^{\infty} M(k, l)\right)=1 \text {, }
	$$
	where
	$$
	M(k, l):=\left\{\mathbf{x} \in[0,1]^d: \mu\left(R\left(f(\mathbf{x}),  \mathbf{r}_n\right)\right) \leq k r_{n, 1} \cdots r_{n, d} \text { for all } n \geq l\right\} \text {. }
	$$
	Note that if $f(\mathbf{x}) \equiv \mathbf{y}$, where $\mathbf{y} \in \Omega_1$, we have $m_d \circ f^{-1}(\Omega_1)=1$. Thus, $f(\mathbf{x}) \equiv \mathbf{y}$, where $\mathbf{y} \in \Omega_1$, satisfies $m_d \circ f^{-1} \ll m_d$ by Lemma \ref{DJ}. If $f=I d_{[0,1]^d}$, $F$ naturally satisfies $m_d \circ f^{-1} \ll m_d$.

	\begin{remark}
		As noted above and under the Condition \ref{121} of Theorem \ref{12}, we can only limit the upper bound of the $\mu$-measure of hyperrectangles whose center lies in $\Omega_1$. Therefore, our primary focus is on unifying quantitative recurrence properties and the shrinking target problem on $\Omega_1$, where $m_d(\Omega_1)=1$. Similarly, under the Condition \ref{131} of Theorem \ref{13}, we are primarily concerned with  unifying quantitative recurrence properties and the shrinking target problem on $V$, where $\mu(V)=1$. For points $\mathbf{y} \in \Omega_1^C$ or $\mathbf{y} \in V^C$, although the density $h$ of $\mu$ satisfies different conditions at these points, according to \cite[Corollary 1]{16}, corresponding zero-full laws still hold. Furthermore, for the case of hyperrectangles, we can get the unity in form in the divergent part.
	\end{remark}
	
	\begin{thm}[{\cite[Corollary 1]{16}}] \label{01y} Let $(X, \mathcal{B}, \mu, T)$ be a measure-preserving dynamical system and $\mathcal{C}$ be a collection of subsets of $X$. Suppose that $\mu$ is $\Sigma$-mixing with respect to $(T, \mathcal{C})$ and let $\left\{E_n\right\}$ be a sequence of subsets in $\mathcal{C}$. Then
		$$
		\mu\left(W\left(T,\left\{E_n\right\}\right)\right)= \begin{cases}0 & \text { if } \sum_{n=1}^{\infty} \mu\left(E_n\right)<\infty \text{,}\\ 1 & \text { if } \sum_{n=1}^{\infty} \mu\left(E_n\right)=\infty \text{.}\end{cases}
		$$
		Here, $W\left(T,\left\{E_n\right\}\right)= \left\{x \in X: T^nx \in E_n \text  { for i.m. } n \in \mathbb{N}\right\}$\text{.}
	\end{thm}
	Let $X=[0,1]^d$, and let $\left\{E_n\right\}$ be a sequence of hyperrectangles with side lengths $\mathbf{r}_n=\left(r_{n, 1}, \ldots, r_{n, d}\right) $, each centered at $\mathbf{y} \in [0,1]^d$. If $\mathbf{y} \in \Omega_1^C$, the density $h(\mathbf{y})$ of $\mu$ is unbounded. From this, we can deduce that if $\sum_{n=1}^{\infty} r_{n, 1} \cdots r_{n, d}=\infty$, then  $\sum_{n=1}^{\infty} \mu\left(E_n\right)=\infty$. By Theorem \ref{01y}, we conclude that
	$$
	\mu\left(W\left(T,\left\{E_n\right\}\right)\right)=1 \text{.}
	$$ 
	This result is consistent with our result in the divergent part of Theorem \ref{12}, but the convergence part differs. Differently, our results are dedicated to the $\mu$-measure of $W\left(T,\left\{E_n\right\}\right)$ is zero or one according to the convergence or divergence of the volume sum $\sum_{n \geq 1} m_d\left(E_n\right)$.
	
	The class of piecewise Lipschitz vector functions $f$ satisfying $m_d \circ f^{-1} \ll m_d$ encompasses, though is not limited to, the following examples:
	
	$\bullet$  $f(\mathbf{x})=\mathbf{x}$, and $f(\mathbf{x}) \equiv \mathbf{y}$, where $\mathbf{y} \in \Omega_1$ defined in Remark \ref{CDXZ}.
	
	$\bullet$  Bijective and  bi-Lipschitz vector functions.  For example,
	$$
	f(\mathbf{x})=\left(-x_1+1, \frac{1}{e-1}\left(e^{x_2}-1\right), \frac{1}{\ln 2} \ln \left(x_3+1\right), kx_4, \cdots, \tan \left(\frac{\pi}{4} x_{d-1}\right), \frac{4}{\pi} \arctan x_d\right),
	$$
	where $k \leq 1$. When $d=2$, $f(\mathbf{x})=\mathbf{x}\left(\begin{array}{ll}\frac{1}{2} & \frac{1}{2} \\ 0 & 1\end{array}\right)$ and $f(\mathbf{x})=\mathbf{x}\left(\begin{array}{ll}0 & 1 \\ 1 & 0\end{array}\right)$.
	
	Based on the above discussions, in some cases, twisted recurrence can manifest as a combination of different types such as the shrinking target problem, recurrence and cross-component recurrence in distinct directions on $[0,1]^d$. 
	
	$\bullet$ Piecewise bijective and  piecewise bi-Lipschitz functions, such as $f(\mathbf{x})= \mathbf{x}\mathbf{A}+\mathbf{b} \bmod 1$, where $\mathbf{A}$ is a $d \times d$ non-singular real matrix and $\mathbf{b}$ is a constant in $[0,1)^d$.
	
	It is worth noting that $f(x)=x^{1 / q}$, where $q>1$, satisfies $m_1 \circ f^{-1} \ll m_1$, but it is not the piecewise Lipschitz function as defined in Definition \ref{pf}.

	For general probability measures $\mu$, due to variety of their density functions $h$, $\mu \circ f^{-1} \ll \mu$ is currently only satisfied by the functions $f(\mathbf{x})=\mathbf{x}$ and $f(\mathbf{x})\equiv\mathbf{y}$, where $\mathbf{y} \in V$ as defined in Theorem \ref{13}. 
	
	Recall that, under the Condition \ref{122} of Theorems \ref{12} and  \ref{13}, the conditions $m_d \circ f^{-1} \ll m_d$ and $\mu \circ f^{-1} \ll \mu$ on $f$ can be removed.
	
	\section{Proof of Theorem \ref{12}} \label{DDDR}
	
	In this section, we fix the collection $\mathcal{C}_1$ of subsets $E \subset[0,1]^d$ satisfying the following bounded property
	$$
	\textbf { (P1) : } \sup _{E \in \mathcal{C}_1} M^{*(d-1)}(\partial E)<4 d c_M+\frac{Kc_M}{1-L^{-(d-1)}} \text{,}
	$$
	where $L$ and $K$ are constants as defined in Definition \ref{pem}, and $c_M$ is a constant defined before Definition \ref{pem}. For any hyperrectangle $R \subset$ $[0,1]^d$, its boundary consists of $2 d$ hyperrectangles whose $(d-1)$-dimensional Lebesgue measure is less than 1. Thus, $M^{*(d-1)}(\partial R) \leq 2 dc_M$, and consequently, all hyperrectangles $R \subset$ $[0,1]^d$ satisfies the bounded property \textbf { (P1)}.
	
	\subsection{The convergence part.} \label{RC} The convergence part can be established by proving the following proposition.
	
	\begin{proposition} \label{PRC}
		If $\sum_{n=1}^{\infty} r_{n, 1} \cdots r_{n, d}<\infty$, then $\mu\left(\mathcal{R}^f\left(\left\{\mathbf{r}_n\right\}\right)\right)=0$.
	\end{proposition}

	For any piecewise Lipschitz vector function $f:[0,1]^d \rightarrow[0,1]^d$, define 
	$$
	E_n:=\left\{\mathbf{x} \in[0,1]^d: T^n \mathbf{x} \in R\left(f(\mathbf{x}), \mathbf{r}_n\right)\right\}\text{.}
	$$
	Clearly, $\mathcal{R}^f\left(\left\{\mathbf{r}_n\right\}\right)=\lim \sup E_n$.
	
	Unlike in the shrinking target case, the sets $E_n$ cannot be expressed in the form $T^{-n} R_n$ for some hyperrectangles $R_n$. To handle this, we follow the approach from \cite[Lemma 3.1]{17} that when considering $E_n$ locally, we study the intersection of $E_n$ with $f^{-1} R\left(f(\mathbf{x}_0),p\mathbf{r}\right)$, where $\mathbf{x}_0 \in [0,1]^d$, $p$ is the Lipschitz constant defined in Definition \ref{pf}, and $\mathbf{r}=\left(r_1, \ldots, r_d\right) \in\left(\mathbb{R}_{\geq 0}\right)^d$. This intersection can be approximated by the preimages of some hyperrectangles under $T$.

	\begin{lemma} \label{JBXZ1}
		Let $R(f(\mathbf{x}_0), p \mathbf{r})$ be a hyperrectangle centered at $f(\mathbf{x}_0) \in[0,1]^d$, with side lengths $p\mathbf{r}=\left(pr_1, \ldots, pr_d\right) \in\left(\mathbb{R}_{\geq 0}\right)^d$. For any subset $A$ of $f^{-1} R\left(f(\mathbf{x}_0), p\mathbf{r}\right)$, 
		$$
		A \cap E_n \subset A \cap T^{-n} R\left(f(\mathbf{x}_0), \mathbf{r}_n+p \mathbf{r}\right)\text{.}
		$$
		Furthermore, if $pr_i \leq r_{n, i}$, for all $1 \leq i \leq d$, then
		$$
		A \cap T^{-n} R\left(f(\mathbf{x}_0), \mathbf{r}_n-p \mathbf{r}\right) \subset A \cap E_n \text{.}
		$$
	\end{lemma}
	
	\begin{proof}
		Fix a point $\mathbf{z} \in A \cap E_n$, then $f(\mathbf{z}) \in R(f(\mathbf{x}_0), p \mathbf{r})$ and $T^n \mathbf{z} \in R\left(f(\mathbf{z}), \mathbf{r}_n\right)$.
		Write $T^n \mathbf{z}=\left(\tilde{z}_1, \ldots, \tilde{z}_d\right)$. For any $1 \leq i \leq d$, 
		$$
		\left|f_i(\mathbf{z})-f_i\left(\mathbf{x}_0\right)\right| \leq p r_i \quad \text { and } \quad\left|\tilde{z}_i-f_i(\mathbf{z})\right|\leq r_{n, i} .
		$$
		By the triangle inequality, 
		$$
		\left|\tilde{z}_i-f_i\left(\mathbf{x}_0\right)\right|\leq \left|\tilde{z}_i-f_i(\mathbf{z})\right|+\left|f_i(\mathbf{z})-f_i\left(\mathbf{x}_0\right)\right| \leq r_{n, i}+p r_i, \quad \text { for } 1 \leq i \leq d .
		$$
		This implies $T^n \mathbf{z} \in R\left(f\left(\mathbf{x}_0\right), \mathbf{r}_n+p \mathbf{r}\right)$. Hence,
		$$
		A \cap E_n \subset A \cap T^{-n} R\left(f\left(\mathbf{x}_0\right), \mathbf{r}_n+p \mathbf{r}\right).
		$$
		For the second inclusion, assume that $p r_i \leq r_{n, i}$ for all $1 \leq i \leq d$, and use similar methods, we have $A \cap T^{-n} R\left(f\left(\mathbf{x}_0\right), \mathbf{r}_n-p \mathbf{r}\right) \subset A \cap E_n$.
		This completes the proof of both inclusions.
	\end{proof}

	Recall that Theorem \ref{ZG} states that if $h$ belongs to $L^q\left(m_d\right)$ for some $q>1$, then for $m_d-a.e.$ $\mathbf{x}$, the $\mu$-measure of a sufficiently small hyperrectangle $R$ centered at $\mathbf{x}$ with sides parallel to the axes is approximately equal to the volume of $R$ multiplied by $h(\mathbf{x})$. Since $m_d \circ f^{-1} \ll m_d$, for $m_d-a.e.$ $\mathbf{x}$, the $\mu$-measure of a sufficiently small coordinate-parallel hyperrectangle $R$ with its center $f(\mathbf{x})$ or $\mathbf{x}$, is also approximately equal to the volume of $R$ multiplied by $h(f(\mathbf{x}))$ or $h(\mathbf{x})$, respectively. We now turn to the proof of Proposition \ref{PRC}.

	\begin{proof} [Proof of Proposition \ref{PRC}]
		By Remark \ref{CDXZ}, we fix one $Z(k, l_1) \cap M(k, l_2)$ with $\mu(Z(k, l_1) \cap M(k, l_2))>0$. Let $n \in \mathbb{N}$ with $n \geq \max (l_1, l_2)$. To obtain the result for the convergence part of hyperrectangles, we need to cover $Z(k, l_1) \cap M(k, l_2)$ with  hyperrectangles. 	We aim to construct a finite family of hyperrectangles $$\left\{R\left(\mathbf{x}_i, \mathbf{r}_n\right): \mathbf{x}_i \in Z(k, l_1) \cap M(k, l_2)\right\}_{i \in \mathcal{I}},$$ covering $Z(k, l_1) \cap M(k, l_2)$, and the `$1 / 2$-scaled up' hyperrectangles $\left\{R\left(\mathbf{x}_i, \mathbf{r}_n / 2\right)\right\}_{i \in \mathcal{I}}$ are pairwise disjoint. 	We start by choosing $\mathbf{x}_1 \in Z(k, l_1) \cap M(k, l_2)$, and let $R\left(\mathbf{x}_1, \mathbf{r}_n\right)$ be a hyperrectangle centered at $\mathbf{x}_1$. Inductively, suppose that $R\left(\mathbf{x}_1, \mathbf{r}_n\right), \ldots, R\left(\mathbf{x}_j, \mathbf{r}_n\right)$ have been defined for some $j \geq 1$. If the union $\cup_{i \leq j} R\left(\mathbf{x}_i, \mathbf{r}_n\right)$ does not cover $Z(k, l_1) \cap M(k, l_2)$, we choose a point $\mathbf{x}_{j+1} \in Z(k, l_1) \cap M(k, l_2) \backslash \cup_{i \leq j} R\left(\mathbf{x}_i, \mathbf{r}_n\right)$. Otherwise, we set $\mathcal{I}=\{1, \ldots, j\}$ and terminate the inductive construction. Next, we need to show that these hyperrectangles $\left\{R\left(\mathbf{x}_i, \mathbf{r}_n / 2\right)\right\}_{i \in \mathcal{I}}$ are pairwise disjoint. For any $i$ and $j$ with $i<j$, by construction, we have $\mathbf{x}_j \notin R\left(\mathbf{x}_i, \mathbf{r}_n\right)$. Therefore, $R\left(\mathbf{x}_i, \mathbf{r}_n / 2\right) \cap R\left(\mathbf{x}_j, \mathbf{r}_n / 2\right)=\emptyset$, proving the disjointness.
		The disjointness of these scaled hyperrectangles implies that 
		\begin{equation}
			\sharp \mathcal{I} \leq\left(r_{n, 1} \cdots r_{n, d}\right)^{-1} \leq \left|\mathbf{r}_n\right|_{\text {min }}^{-d}.
			\label{Eq.3.1}
		\end{equation}
		Take $\mathbf{z} \in R\left(\mathbf{x}_i, \mathbf{r}_n\right)$, which implies that $\left|\mathbf{z}-\mathbf{x}_i\right| \leq \left|\mathbf{r}_n\right|$.  Since $f$ is  piecewise Lipschitz vector function, we obtain $\left|f(\mathbf{z})-f\left(\mathbf{x_i}\right)\right| \leq p\left|\mathbf{z}-\mathbf{x}_{\mathbf{i}}\right| \leq p\left|\mathbf{r}_n\right|$, that is $f(\mathbf{z}) \in B\left(f\left(\mathbf{x}_i\right), p \left|\mathbf{r}_n\right|\right)$. Here $B\left(f\left(\mathbf{x}_i\right), p \left|\mathbf{r}_n\right|\right)$ is a hypercube (ball), centered at $f\left(\mathbf{x}_i\right)$ with radius $p \left|\mathbf{r}_n\right|$. Thus, we have the inclusion
		$$
		R\left(\mathbf{x}_i, \mathbf{r}_n\right) \subset f^{-1} B\left(f\left(\mathbf{x}_i\right), p \left|\mathbf{r}_n\right|\right).
		$$ 
		By Lemma \ref{JBXZ1}, we obtain the inclusion 
		$$
		Z(k, l_1) \cap M(k, l_2) \cap E_n \subset \bigcup_{i \in \mathcal{I}} R\left(\mathbf{x}_i, \mathbf{r}_n\right) \cap E_n \subset \bigcup_{i \in \mathcal{I}} R\left(\mathbf{x}_i, \mathbf{r}_n\right) \cap T^{-n} B\left(f\left(\mathbf{x}_i\right),(p+1) \left|\mathbf{r}_n\right|\right).
		$$
		By the summable-mixing property of $\mu$, and using \eqref{Eq.1.2}, \eqref{Eq.3.1} as well as the definition of $Z(k, l_1) \cap M(k, l_2)$, we have
		$$
		\begin{aligned}
			&\mu\left(Z(k, l_1) \cap M(k, l_2) \cap E_n\right)  \leq \sum_{i \in \mathcal{I}} \mu\left(R\left(\mathbf{x}_i, \mathbf{r}_n\right) \cap T^{-n} B\left(f\left(\mathbf{x}_i\right),(p+1) \left|\mathbf{r}_n\right|\right)\right) \\
			& \leq \sum_{i \in \mathcal{I}}\left(\mu\left(R\left(\mathbf{x}_i, \mathbf{r}_n\right)\right)+\phi(n)\right) \mu\left(B\left(f\left(\mathbf{x}_i\right),(p+1) \left|\mathbf{r}_n\right|\right)\right) \\
			& \leq\left|\mathbf{r}_n\right|_{\text {min }}^{-d} \cdot \left(k r_{n, 1} \cdots r_{n, d}+\phi(n)\right) \cdot k (p+1)^d \left|\mathbf{r}_n\right|^d \\
			& =k(p+1)^d \left(\frac{\left|\mathbf{r}_n\right|}{\left|\mathbf{r}_n\right|_{\text {min }}}\right)^d \left(k r_{n, 1} \cdots r_{n, d}+\phi(n)\right).
		\end{aligned}
		$$
		Since $\sum_{n=1}^{\infty} r_{n, 1} \cdots r_{n, d}<\infty$ and the sequence $\left\{\frac{\left|\mathbf{r}_n\right|}{\left|\mathbf{r}_n\right|_{\text {min }}}\right\}$ is bounded, we conclude that
		$$
		\sum_{n=1}^{\infty} \mu\left(\left(Z(k, l_1) \cap M(k, l_2) \cap E_n\right)<\infty\right. \text{.}
		$$
		By Borel-Cantelli lemma, $\mu\left(\left(Z(k, l_1) \cap M(k, l_2) \cap \mathcal{R}^f\left(\left\{\mathbf{r}_n\right\}\right.\right)=0\right.$. Finally, it follows from Remark \ref{CDXZ} that
		$$
		\mu\left(\mathcal{R}^f\left(\left\{\mathbf{r}_n\right\}\right)\right)=0.
		$$
		
		Under the Condition \ref{122} of Theorem \ref{12}, it suffices to replace $Z(k, l_1) \cap M(k, l_2)$ with $[0,1]^d$. Moreover, it is unnecessary to assume that $f$ satisfies $m_d \circ f^{-1} \ll m_d$, and the convergence part of Theorem \ref{12} also holds.
		
		Furthermore, if $f(\mathbf{x})=\left(f_{1^{'}}\left(x_1\right), f_{2^{'}}\left(x_2\right), \cdots, f_{d^{'}}\left(x_d\right)\right)$, we will list the differences in the proof so that the convergence part holds without the condition that the sequence $\left\{\frac{\left|\mathbf{r}_n\right|}{\left|\mathbf{r}_n\right|_{\text {min }}}\right\}$ is bounded.
		Take $\mathbf{z} \in R\left(\mathbf{x}_i, \mathbf{r}_n\right)$, which implies that $\left|z_j-x_{i, j}\right| \leq r_{n, j}$ for all $1 \leq j \leq d$. Since $f$ is piecewise Lipschitz vector function such that $f_{j^{'}}: [0,1] \rightarrow [0,1]$ is piecewise Lipschitz function, we have $\left|f_{j^{'}}\left(z_j\right)-f_{j^{'}}\left(x_{i, j}\right)\right| \leq p\left|z_j-x_{i, j}\right| \leq p r_{n, j}$, for each $j=j^{'}$. Thus, we have the inclusion
		$$
		R\left(\mathbf{x}_i, \mathbf{r}_n\right) \subset f^{-1} R\left(f\left(\mathbf{x}_i\right), p \mathbf{r}_n\right).
		$$
		By Lemma \ref{JBXZ1}, we obtain the inclusion	
		$$
		Z\left(k, l_1\right) \cap M\left(k, l_2\right) \cap E_n \subset \bigcup_{i \in \mathcal{I}} R\left(\mathbf{x}_i, \mathbf{r}_n\right) \cap E_n \subset \bigcup_{i \in \mathcal{I}} R\left(\mathbf{x}_i, \mathbf{r}_n\right) \cap T^{-n} R\left(f\left(\mathbf{x}_i\right),(p+1) \mathbf{r}_n\right).
		$$
		Finally, based on $\sharp \mathcal{I} \leq\left(r_{n, 1} \cdots r_{n, d}\right)^{-1}$, $\mu\left(R\left(f\left(\mathbf{x}_i\right),(p+1) \mathbf{r}_n\right)\right) \leq k(p+1)^d r_{n, 1} \cdots r_{n, d}$ and the same reason, we complete the convergence part.
	\end{proof}

	\subsection{The divergence part.}  \label{RD}
	In this part, we assume that $ \sum_{n\ge 1} r_{n,1}\cdots r_{n,d}=\infty $. We note that
	$$
	\sum_{n: \forall i, r_{n, i}>n^{-2}} r_{n, 1} \cdots r_{n, d}=\infty \Longleftrightarrow \sum_{n=1}^{\infty} r_{n, 1} \cdots r_{n, d}=\infty
	$$
	and 
	\begin{equation}
		\limsup _{n: \forall i, r_{n, i}>n^{-2}} E_n \subset \limsup _{n \rightarrow \infty} E_n.
		\label{Eq.3.2}
	\end{equation}
	To prove the divergence part, we aim to show that the left side of \eqref{Eq.3.2} is a full $\mu$-measure set. Without loss of generality, we further assume that for all $ n\in\N $, either $ r_{n,i} $ is greater than $ n^{-2} $ for all $ 1\le i\le d $, or equal to $ 0 $ for all $ 1\le i\le d $.
	
	Next, following the approach outlined in \cite{7}, we construct a sequence of auxiliary sets $\hat{E}_n$.

	\begin{lemma}
		For any $\mathbf{x} \in[0,1]^d$ and $n \in \mathbb{N}$, and for a vector function $f:[0,1]^d \rightarrow[0,1]^d$, $\exists l_n(\mathbf{x}) \in \mathbb{R}_{\geq 0}$ such that
		$$
		\mu\left(R\left(f(\mathbf{x}), l_n(\mathbf{x}) \mathbf{r}_n\right)\right)=r_{n, 1} \cdots r_{n, d}.
		$$
	\end{lemma}
	
	\begin{proof}
		Fix $n \in \mathbb{N}$	and define a function $g: \mathbb{R}_{\geq 0} \rightarrow [0,1]$ by
		$
		g(t)=\mu\left(R\left(f(\mathbf{x}), t \mathbf{r}_n\right)\right)
		$. Clearly, $g(0)=0$ and $g(\infty)=1$.
		
		Next, we will prove that $g$ is a continuous function; that is, $\forall \varepsilon>0$, we need to find $\delta(\varepsilon)>0$, such that for $\left|t-t_0\right| < \delta(\varepsilon)$, $\left|g(t)-g\left(t_0\right)\right|<\varepsilon$. 
		
		By Hölder's inequality,  we estimate $\left|g(t)-g\left(t_0\right)\right|$ as follows.
		\begin{equation}
			\begin{aligned}
				\left|g(t)-g\left(t_0\right)\right| & =\left|\mu\left(R\left(f(\mathbf{x}), t \mathbf{r}_n\right)\right)-\mu\left(R\left(f(\mathbf{x}), t_0 \mathbf{r}_n\right)\right)\right| \\
				& \leq \int_{R\left(f(\mathbf{x}),\left(t_0+\delta(\varepsilon)\right) \mathbf{r}_n\right) \backslash R\left(f(\mathbf{x}), t_0 \mathbf{r}_n\right)} \chi(\mathbf{x}) d \mu(\mathbf{x}) \\
				& =\int \chi_{R\left(f(\mathbf{x}),\left(t_0+\delta(\varepsilon)\right) \mathbf{r}_n\right) \backslash R\left(f(\mathbf{x}), t_0 \mathbf{r}_n\right)} \cdot h(\mathbf{x}) \mathrm{~d} m_d(\mathbf{x}) \\
				& \leq \|h\|_q \cdot m_d\left(R\left(f(\mathbf{x}),\left(t_0+\delta(\varepsilon)\right) \mathbf{r}_n\right) \backslash R\left(f(\mathbf{x}), t_0 \mathbf{r}_n\right)\right)^{1-1 / q},
			\end{aligned}
			\label{Eq.3.3}
		\end{equation}
		where $\|h\|_q:=\left(\int|h|^q \mathrm{~d} m_d\right)^{1 / q}$ is the $L^q$-norm of $h$.
		
		Since $R\left(f(\mathbf{x}), (t_0+\delta(\varepsilon))\mathbf{r}_n\right) \backslash R\left(f(\mathbf{x}), t_0\mathbf{r}_n\right)$ is contained in $2 d$ hyperrectangles, each with a volume less than $\delta(\varepsilon) \cdot \left|\mathbf{r}_n\right|$, and by \eqref{Eq.3.3}, we have
		$$
		\left|g(t)-g\left(t_0\right)\right| \leq\
		\|h\|_q \cdot \left(\delta(\varepsilon)\left|\mathbf{r}_n\right|\right)^{1-1 / q} <\varepsilon, 
		$$
		where $\delta(\varepsilon)=(\varepsilon /\|h\|_q)^{q / (q-1)} \cdot 1/ \left|\mathbf{r}_n\right|$.
		Then, the remaining proof of the lemma is completed by the intermediate value theorem of continuous functions.
	\end{proof}
	
	Let $\xi_n(\mathbf{x}):=l_n(\mathbf{x}) \mathbf{r}_n \in\left(\mathbb{R}_{\geq 0}\right)^d$, define
	$$
	\hat{E}_n:=\left\{\mathbf{x} \in[0,1]^d: T^n \mathbf{x} \in R\left(f(\mathbf{x}), \xi_n(\mathbf{x})\right)\right\} \quad \text { and } \quad \hat{\mathcal{R}}^f\left(\left\{\mathbf{r}_n\right\}\right):=\lim \sup \hat{E}_n.
	$$
	
	\begin{lemma} \label{l34}
		If the set $\hat{\mathcal{R}}^f\left(\left\{\mathbf{r}_n\right\}\right)$ is of full $\mu$-measure, then the set $\mathcal{R}^f\left(\left\{\mathbf{r}_n\right\}\right)$ is of full $\mu$-measure.
	\end{lemma}
	\begin{proof}
		We can find a sequence of positive numbers $\left\{a_n\right\}$ such that
		$$
		\sum_{n=1}^{\infty} \frac{r_{n, 1} \cdots r_{n, d}}{a_n^d}=\infty \quad \text { and } \quad \lim _{n \rightarrow \infty} a_n=\infty.
		$$
		We define the scaled sequence $\tilde{\mathbf{r}}_n=a_n^{-1} \mathbf{r}_n$. From the assumption that $\hat{\mathcal{R}}^f\left(\tilde{\mathbf{r}}_n\right)$  is of full $\mu$-measure, for $\mu-a.e.$ $\mathbf{x}$, we have
		$$
		T^n \mathbf{x} \in R\left(f(\mathbf{x}), \tilde{\xi}_n(\mathbf{x})\right) \text { for i.m. } n \in \mathbb{N},
		$$
		where $\tilde{\xi}_n(\mathbf{x}):=\tilde{l}_n(\mathbf{x}) \tilde{\mathbf{r}}_n$ and $\tilde{l}_n(\mathbf{x}) \in \mathbb{R}_{\geq 0}$ is the positive number such that
		$$
		\mu\left(R\left(f(\mathbf{x}), \tilde{l}_n(\mathbf{x}) \tilde{\mathbf{r}}_n\right)\right)=r_{n, 1} \cdots r_{n, d} / a_n^d.
		$$
		By Zygmund differentiation theorem (see Theorem \ref{ZG}) and $m_d \circ f^{-1} \ll m_d$, for $m_d$-almost every $\mathbf{x}$, we have 
		\begin{equation}
			0 \leq \lim _{n \rightarrow \infty} \frac{\mu\left(R\left(f(\mathbf{x}), \mathbf{r}_n\right)\right)}{2^d \cdot r_{n, 1} \cdots r_{n, d}}=h(f(\mathbf{x})).
			\label{Eq.3.3.1}
		\end{equation}
		Under the Condition \ref{122} of Theorem \ref{12}, for any  $\mathbf{x}$, \eqref{Eq.3.3.1} naturally holds.
		
		For any such $\mathbf{x}$ with $h(f(\mathbf{x}))>0$, since $\lim _{n \rightarrow \infty} a_n=\infty$, there exists an $N(\mathbf{x})$ such that for all $n \geq N(\mathbf{x})$,
		$$
		\mu\left(R\left(f(\mathbf{x}), \tilde{\xi}_n(\mathbf{x})\right)\right)=r_{n, 1} \cdots r_{n, d} / a_n^d< 2^{d-1} h(f(\mathbf{x})) \cdot r_{n, 1} \cdots r_{n, d}.
		$$
		Increasing the integer $N(\mathbf{x})$ if necessary, we have
		$$
		\mu\left(R\left(f(\mathbf{x}), \tilde{\xi}_n(\mathbf{x})\right)\right)<\mu\left(R\left(f(\mathbf{x}), \mathbf{r}_n\right)\right), \text { for all } n \geq N(\mathbf{x}) \text {. }
		$$
		If $h(f(\mathbf{x}))=0$, $\mu\left(R\left(f(\mathbf{x}), \tilde{\xi}_n(\mathbf{x})\right)\right)=\mu\left(R\left(f(\mathbf{x}), \mathbf{r}_n\right)\right).$
		Therefore,
		$$
		R\left(f(\mathbf{x}), \tilde{\xi}_n(\mathbf{x})\right) \subset R\left(f(\mathbf{x}), \mathbf{r}_n\right).
		$$
		Thus, combining the condition, we conclude that for $\mu-a.e.$ $\mathbf{x}$,
		$$
		T^n \mathbf{x} \in R\left(f(\mathbf{x}), \mathbf{r}_n\right) \text { for i.m. } n \in \mathbb{N} \text {, }
		$$
		which completes the lemma proof.
	\end{proof}

	The next lemma characterizes the local structures of $ \hat{E}_n $, which resemble those of $ E_n $.
	
	\begin{lemma} \label{l35}
		Let $R(f(\mathbf{x}_0), p \mathbf{r})$ be a hyperrectangle centered at $f(\mathbf{x}_0) \in[0,1]^d$ with $\mathbf{r}=\left(r_1, \ldots, r_d\right) \in\left(\mathbb{R}_{\geq 0}\right)^d$, and let $t=\max _{1 \leq i \leq d}\left(p r_i / r_{n, i}\right)$. For any subset $A$ of $f^{-1} R\left(f(\mathbf{x}_0),  p \mathbf{r} \right)$, 
		$$
		A \cap \hat{E}_n \subset A \cap T^{-n} R\left(f(\mathbf{x}_0), \xi_n(\mathbf{x}_0)+2 t \mathbf{r}_n\right)\text{.}
		$$
		Furthermore, if $2t \leq l_n(\mathbf{x}_0)$, then
		$$
		A \cap T^{-n} R\left(f(\mathbf{x}_0), \xi_n(\mathbf{x_0})-2 t \mathbf{r}_n\right) \subset A \cap \hat{E}_n \text{.}
		$$
		
	\end{lemma}
	
	\begin{proof}
		Given that $p r_i=p r_i / r_{n, i} \cdot r_{n, i} \leq t r_{n, i}$, for any $\mathbf{z} \in A$, we have
		$$
		f(\mathbf{z}) \in  R(f(\mathbf{x}_0), p \mathbf{r}) \subset R\left(f(\mathbf{x}_0), t \mathbf{r}_n\right).
		$$
		Using the triangle inequality, we have 
		$$
		R\left(f(\mathbf{z}), \xi_n(\mathbf{x_0})-t \mathbf{r}_n\right) \subset R\left(f(\mathbf{x}_0), \xi_n(\mathbf{x}_0)\right) \subset R\left(f(\mathbf{z}), \xi_n(\mathbf{x}_0)+t \mathbf{r}_n\right).
		$$
		Hence,
		$$
		\begin{aligned}	
			\mu\left(R\left(f(\mathbf{z}), \xi_n(\mathbf{x}_0)-t \mathbf{r}_n\right)\right) \leq \mu\left(R\left(f(\mathbf{x}_0), \xi_n(\mathbf{x}_0)\right)\right)&=r_{n, 1} \cdots r_{n, d} \\ & \leq \mu\left(R\left(f(\mathbf{z}), \xi_n(\mathbf{x}_0)+t \mathbf{r}_n\right)\right).
		\end{aligned}
		$$
		From the definitions of $l_n(\mathbf{z})$ and $\xi_n(\mathbf{z})$, the above inequalities imply that
		$$
		l_n(\mathbf{x}_0)-t \leq l_n(\mathbf{z}) \leq l_n(\mathbf{x}_0)+t,
		$$
		thus
		$$
		R\left(f(\mathbf{z}), \xi_n(\mathbf{x}_0)-t \mathbf{r}_n\right) \subset R\left(f(\mathbf{z}), \xi_n(\mathbf{z})\right) \subset R\left(f(\mathbf{z}), \xi_n(\mathbf{x}_0)+t \mathbf{r}_n\right).
		$$
		If $\mathbf{z}$ also belongs to $\hat{E}_n$, then $T^n \mathbf{z} \in R\left(f(\mathbf{z}), \xi_n(\mathbf{z})\right)$. This leads to
		$$
		T^n \mathbf{z} \in R\left(f(\mathbf{z}), \xi_n(\mathbf{z})\right) \subset R\left(f(\mathbf{z}), \xi_n(\mathbf{x}_0)+t \mathbf{r}_n\right) \subset R\left(f(\mathbf{x}_0), \xi_n(\mathbf{x}_0)+2 t \mathbf{r}_n\right).
		$$
		Hence, $\mathbf{z} \in T^{-n} R\left(f(\mathbf{x}_0), \xi_n(\mathbf{x}_0)+2 t \mathbf{r}_n\right)$. This shows that,
		$$
		A \cap \hat{E}_n \subset A \cap T^{-n} R\left(f\left(\mathbf{x}_0\right), \xi_n(\mathbf{x}_0)+2 t \mathbf{r}_n\right).
		$$
		The second part of the lemma follows similarly, completing the proof.
	\end{proof}

	Next, we need to estimate the $\mu$-measure of the hyperrectangles appearing in Lemma \ref{l35}.
	The proof of the following lemma follows the same proof as \cite[Lemma 2.5]{7}, with the only modification being the use of hyperrectangles centered at $f(\mathbf{x})$.

	\begin{lemma}[{\cite[Lemma 2.5]{7}}] \label{l36}
		Let $s=1-1 / q$. For any $\mathbf{x} \in[0,1]^d$ and $\mathbf{r}=\left(r_1, \ldots, r_d\right) \in\left(\mathbb{R}_{\geq 0}\right)^d$, 
		$$
		\mu\left(R\left(f(\mathbf{x}), \xi_n(\mathbf{x})+\mathbf{r}\right)\right) \leq r_{n, 1} \cdots r_{n, d}+c_1 \cdot \max _{1 \leq i \leq d} r_i^s   $$
		and
		$$
		\mu\left(R\left(f(\mathbf{x}), \xi_n(\mathbf{x})-\mathbf{r}\right)\right) \geq r_{n, 1} \cdots r_{n, d}-c_1 \cdot \max _{1 \leq i \leq d} r_i^s\text{,} 
		$$
		where $c_1=2 d\|h\|_q$.
	\end{lemma}
	
	\begin{remark}
		Under the Condition \ref{122} of Theorem \ref{12}, Lemma \ref{l36} also holds with $s=1$.
	\end{remark}
	
	The method we present next, like most approaches used to establish full $\mu$-measure results, relies on the local Lebesgue density theorem and Paley-Zygmund inequality. We begin by stating the following estimations.
	
	\begin{lemma} \label{l37}
		Let $B \subset[0,1]^d$ be a ball. Then, for sufficiently large $n$,
		$$
		\frac{1}{2} \mu(B) \cdot r_{n, 1} \cdots r_{n, d} \leq \mu\left(B \cap \hat{E}_n\right) \leq 2 \mu(B) \cdot r_{n, 1} \cdots r_{n, d}.
		$$
	\end{lemma}
	
	\begin{proof}
		If $\mathbf{r}_n=\mathbf{0}$, then $\mu\left(B \cap \hat{E}_n\right)=0$ and the lemma holds trivially.
		Recall that a ball here is with respect to the maximum norm so it corresponds to a Euclidean hypercube. Denote the radius of $B$ by $r_0$.
		
		Partition $B$ into $r_0^d \cdot \phi(n)^{-\frac{1}{2}}$ balls with radius $r:=\phi(n)^{\frac{1}{2 d}}$. The collection of these balls is denoted by
		$$
		\left\{B\left(\mathbf{x}_i, r\right): 1 \leq i \leq r_0^d \cdot \phi(n)^{-\frac{1}{2}}\right\}.
		$$
		Using the $p$-Lipschitz condition, we have
		$$
		B\left(\mathbf{x}_i, r\right) \subset f^{-1} B\left(f\left(\mathbf{x}_i\right), p r\right).
		$$
		Let $t=\max _{1 \leq i \leq d}\left(pr / r_{n, i}\right)$, and $2t \leq l_n(\mathbf{x}_i)$ is clear. By Lemma \ref{l35}, we obtain the inclusion
		$$
		B \cap \hat{E}_n=\bigcup_{i \leq r_0^d \cdot \phi(n)^{-\frac{1}{2}}} B\left(\mathbf{x}_i, r\right) \cap \hat{E}_n \supset \bigcup_{i \leq r_0^d \cdot \phi(n)^{-\frac{1}{2}}} B\left(\mathbf{x}_i, r\right) \cap T^{-n} R\left(f\left(\mathbf{x}_i\right), \xi_n(\mathbf{x}_i)-2 t \mathbf{r}_n\right).
		$$
		Since $\mu$ is polynomial-mixing with respect to $(T, \mathcal{C})$, and by Lemma \ref{l36}, we have
		$$
		\begin{aligned}
			\mu\left(B \cap \hat{E}_n\right) & \geq \sum_{i \leq r_0^d \cdot \phi(n)^{-\frac{1}{2}}} \mu\left(B\left(\mathbf{x}_i, r\right) \cap T^{-n} R\left(f(\mathbf{x}_i), \xi_n\left(\mathbf{x}_i\right)-2 t \mathbf{r}_n\right)\right) \\
			& \geq \sum_{i \leq r_0^d \cdot \phi(n)^{-\frac{1}{2}}}\left(\mu\left(B\left(\mathbf{x}_i, r\right)\right)-\phi(n)\right) \mu\left(R\left(f(\mathbf{x}_i), \xi_n\left(\mathbf{x}_i\right)-2 t \mathbf{r}_n\right)\right) \\
			& \geq \sum_{i \leq r_0^d \cdot \phi(n)^{-\frac{1}{2}}}\left(\mu\left(B\left(\mathbf{x}_i, r\right)\right)-\phi(n)\right)\left(r_{n, 1} \cdots r_{n, d}-c_1(2 t)^s\right) \\
			& =\left(\mu(B)- r_0^d \phi(n)^{\frac{1}{2 }}\right)\left(r_{n, 1} \cdots r_{n, d}-c_1(2 t)^s\right),
		\end{aligned}
		$$
		where we use $\lim _{n \rightarrow \infty}\left|\mathbf{r}_n\right|=0<1$ in the third inequality. Since $\mathbf{r}_n \neq \mathbf{0}$, and based on the assumption about the sequence $\left\{\mathbf{r}_n\right\}$, we have $r_{n, i}>n^{-2}$ for all $1 \leq i \leq d$. Hence, 
		$$
		t^s=\left(\max _{1 \leq i \leq d} \frac{pr}{r_{n, i}}\right)^s \leq  p^s n^{2s}  \phi(n)^{\frac{s}{2 d}}< \frac{1}{c_1 2^{s+2}}r_{n, 1} \cdots r_{n, d}.
		$$
		By substituting the upper bound for $t^s$, we deduce that for all sufficiently large $n$,
		$$
		\mu\left(B \cap \hat{E}_n\right) \geq \frac{1}{2} \mu(B) \cdot r_{n, 1} \cdots r_{n, d}.
		$$
		The upper estimation for $\mu\left(B \cap \hat{E}_n\right)$ follows similarly by
		$$
		B \cap \hat{E}_n=\bigcup_{i \leq r_0^d \cdot \phi(n)^{-\frac{1}{2}}} B\left(\mathbf{x}_i, r\right) \cap \hat{E}_n \subset \bigcup_{i \leq r_0^d \cdot \phi(n)^{-\frac{1}{2}}} B\left(\mathbf{x}_i, r\right) \cap T^{-n} R\left(f(\mathbf{x}_i), \xi_n\left(\mathbf{x}_i\right)+2 t \mathbf{r}_n\right).
		$$	\end{proof}
	
	Next, we estimate the $\mu$-measure of $B \cap \hat{E}_m \cap \hat{E}_n$ with $m<n$.
	If $\mathbf{r}_m=\mathbf{0}$ or $\mathbf{r}_n=\mathbf{0}$, then $\mu\left(\hat{E}_m\right)=\mu\left(\hat{E}_n\right)=0$. Therefore, for any ball $B$, we have $\mu\left(B \cap \hat{E}_m \cap \hat{E}_n\right)=0$, then the following lemmas hold trivially. We assume that neither $\mathbf{r}_m$ nor $\mathbf{r}_n$ is $\mathbf{0}$. Then, according to the assumptions at the beginning of Section \ref{RD}, we have $r_{m, i}>m^{-2}$ and $r_{n, i}>n^{-2}$ for all $1 \leq i \leq d$. The following lemma established in \cite[Lemma 2.8]{7} serves as the pivotal tool for estimating $\mu\left(B \cap \hat{E}_m \cap \hat{E}_n\right)$.

	\begin{lemma}[{\cite[Lemma 2.8]{7}}] \label{l38}
		Let $T:[0,1]^d \rightarrow[0,1]^d$ be a piecewise expanding map. For any $n \geq 1$, any $J_n \in \mathcal{F}_n$, and any hyperrectangles $R_1, R_2 \subset[0,1]^d$, 
		$$
		M^{*(d-1)}\left(\partial\left(J_n \cap R_1 \cap T^{-n} R_2\right)\right) \leq 4 d c_M+\frac{Kc_M}{1-L^{-(d-1)}},
		$$
		where $L$ and $K$ are constants defined in  Definition \ref{pem}, and $c_M$ is a constant defined before Definition \ref{pem}.
	\end{lemma}

	 For precision and convenience, we introduce the extended form of $\phi(n)=c e^{-\tau n}$ as
	$$\phi_1(x): =c e^{-\tau x},$$
	where $c$ and $\tau$ are defined in Definition \ref{mixing}. Clearly, $\phi_1: \mathbb{R}_{+} \rightarrow (0,1)$ is a continuous function, and its inverse function $\phi_1^{-1}(x)=1 / \tau \log (c / x)$ exists. 
	
	\begin{lemma} \label{l39}
		Let $B \subset[0,1]^d$ be a ball. Then, there exists a constant $c_2$ such that for all large integers $m$ and $n$ with $\phi_1^{-1}\left(p^{-2 d} n^{-4d(d/s+1)}\right) \leq m<n$, 
		$$
		\mu\left(B \cap \hat{E}_m \cap \hat{E}_n\right) \leq c_2 \mu(B)\left(r_{m, 1} \cdots r_{m, d}+\phi(n-m)\right) \cdot r_{n, 1} \cdots r_{n, d}.
		$$
	\end{lemma}
	\begin{proof}
		Let $r_0$ be the radius of $B$. Partition $B$ into $r_0^d \cdot \phi(m)^{-\frac{1}{2}}$ balls with radius $r:=\phi(m)^{\frac{1}{2 d}}$. The collection of these balls is denoted by
		$$
		\left\{B\left(\mathbf{x}_i, r\right): 1 \leq i \leq r_0^d \cdot \phi(m)^{-\frac{1}{2}}\right\}.
		$$
		By the $p$-Lipschitz condition, we have
		$$
		B\left(\mathbf{x}_i, r\right) \subset f^{-1} B\left(f\left(\mathbf{x}_i\right), p r\right) .
		$$
		Let $t_m=\max _{1 \leq i \leq d}\left(pr / r_{m, i}\right)$ and $t_n=\max _{1 \leq i \leq d}\left(pr / r_{n, i}\right)$. For each ball $B\left(\mathbf{x}_i, r\right)$, by Lemma \ref{l35}, we have 
		$$
		\begin{aligned}
			&	B \cap \hat{E}_m \cap \hat{E}_n  \\ & \subset \bigcup_{i \leq r_0^d \phi(m)^{-\frac{1}{2}}} B\left(\mathbf{x}_i, r\right) \cap T^{-m} R\left(f(\mathbf{x}_i), \xi_m\left(\mathbf{x}_i\right)+2 t_m \mathbf{r}_m\right) \cap T^{-n} R\left(f(\mathbf{x}_i), \xi_n\left(\mathbf{x}_i\right)+2 t_n \mathbf{r}_n\right).		
		\end{aligned}
		$$
		For notational simplicity, we write
		$$
		R_i(m)=R\left(f(\mathbf{x}_i), \xi_m\left(\mathbf{x}_i\right)+2 t_m \mathbf{r}_m\right) \quad \text { and } \quad R_i(n)=R\left(f(\mathbf{x}_i), \xi_n\left(\mathbf{x}_i\right)+2 t_n \mathbf{r}_n\right).
		$$
		By Lemma \ref{l38}, the sets $J_{n-m} \cap R_i(m) \cap T^{-(n-m)} R_i(n)$ satisfy the bounded property \textbf { (P1)}. Applying the polynomial-mixing property of $\mu$ and Lemma \ref{l36}, we have
		$$
		\begin{aligned}
			& \mu\left(B\left(\mathbf{x}_i, r\right) \cap T^{-m} R_i(m) \cap T^{-n} R_i(n)\right) \\
			= & \sum_{J_{n-m} \in \mathcal{F}_{n-m}} \mu\left(B\left(\mathbf{x}_i, r\right) \cap T^{-m}\left(J_{n-m} \cap R_i(m) \cap T^{-(n-m)} R_i(n)\right)\right) \\
			\leq & \sum_{J_{n-m} \in \mathcal{F}_{n-m}}\left(\mu\left(B\left(\mathbf{x}_i, r\right)\right)+\phi(m) \right) \mu\left(J_{n-m} \cap R_i(m) \cap T^{-(n-m)} R_i(n)\right) \\
			= & \left(\mu\left(B\left(\mathbf{x}_i, r\right)\right)+\phi(m)\right) \mu\left(R_i(m) \cap T^{-(n-m)} R_i(n)\right) \\
			\leq & \left(\mu\left(B\left(\mathbf{x}_i, r\right)\right)+\phi(m)\right)\left(\mu\left(R_i(m)\right)+\phi(n-m)\right) \mu\left(R_i(n)\right) \\
			\leq & 
			\left(\mu\left(B\left(\mathbf{x}_i, r\right)\right)+\phi(m)\right)\left(r_{m, 1} \cdots r_{m, d}+c_1\left(2 t_m\right)^s+\phi(n-m)\right)\left(r_{n, 1} \cdots r_{n, d}+c_1\left(2 t_n\right)^s\right).
		\end{aligned}
		$$
		Summing over $i \leq r_0^d \phi(m)^{-\frac{1}{2}}$, we have
		$$
		\begin{aligned}
			&\mu\left(B \cap \hat{E}_m \cap \hat{E}_n\right) \\ &\leq
			\left(\mu\left(B\right)+r_0^d\phi(m)^{\frac{1}{2}}\right)\left(r_{m, 1} \cdots r_{m, d}+c_1\left(2 t_m\right)^s+\phi(n-m)\right)\left(r_{n, 1} \cdots r_{n, d}+c_1\left(2 t_n\right)^s\right).
		\end{aligned}
		$$
		Since $r_{m, i}>m^{-2}$ and $r_{n, i}>n^{-2}$, we have
		$$
		t^s_m=\max _{1 \leq i \leq d} \frac{p^s \phi(m)^{\frac{s}{2 d}}}{r_{m, i}^s} \leq p^s m^{2s}\phi(m)^{\frac{s}{2 d}} \leq r_{m, 1} \cdots r_{m, d}.
		$$
		From $\phi_1^{-1}\left(p^{-2 d} n^{-4d(d/s+1)}\right) \leq m<n$ that
		$$
		t_n^s=\left(\max _{1 \leq i \leq d} \frac{p \phi(m)^{\frac{1}{2 d }}}{r_{n, i}}\right)^s \leq p^s n^{2s} \phi(m)^{\frac{s}{2 d}} \leq n^{-2d} < r_{n, 1} \cdots r_{n, d}.
		$$
		Thus, there exists a constant $c_2$ such that for any sufficiently large integers $m$ and $n$ with $\phi_1^{-1}\left(p^{-2 d} n^{-4d(d/s+1)}\right) \leq m<n$, 
		$$
		\mu\left(B \cap \hat{E}_m \cap \hat{E}_n\right) \leq c_2 \mu(B)\left(r_{m, 1} \cdots r_{m, d}+\phi(n-m)\right) \cdot r_{n, 1} \cdots r_{n, d}.
		$$
	\end{proof}
	
	The following lemma was stated and used in \cite[Lemma 2.10]{7}. We provide the proof here, as our case differs slightly.
	
	\begin{lemma} \label{l310}
		Let $T:[0,1]^d \rightarrow[0,1]^d$ be a piecewise expanding map. Then, there exists a constant $c_3$ such that for any integers $m$ and $n$ with $m < \phi_1^{-1}\left(p^{-2 d} n^{-4d(d/s+1)}\right)$, and any hyperrectangles $R_1, R_2, R_3 \subset[0,1]^d$, 
		$$
		\mu\left(R_1 \cap T^{-m} R_2 \cap T^{-n} R_3\right) \leq\left(\mu\left(R_1 \cap T^{-m} R_2\right)+c_3 \phi(n)^{\frac{1}{2}}\right) \mu\left(R_3\right).
		$$
		
	\end{lemma}

	\begin{proof} It is clear that
		$$
		\mu\left(R_1 \cap T^{-m} R_2 \cap T^{-n} R_3\right)=\sum_{J_m \in \mathcal{F}_m} \mu\left(J_m \cap R_1 \cap T^{-m} R_2 \cap T^{-n} R_3\right).
		$$
		By Lemma \ref{l38}, the sets $J_m \cap R_1 \cap T^{-m} R_2$ satisfy the bounded property \textbf {(P1)}. Now, applying the exponential-mixing property of $\mu$, we have
		$$
		\begin{aligned}
			& \sum_{J_m \in \mathcal{F}_m} \mu\left(J_m \cap R_1 \cap T^{-m} R_2 \cap T^{-n} R_3\right) \\
			\leq & \sum_{J_m \in \mathcal{F}_m}\left(\mu\left(J_m \cap R_1 \cap T^{-m} R_2\right)+\phi(n)\right) \mu\left(R_3\right) \\
			= & \left(\mu\left(R_1 \cap T^{-m} R_2\right)+\# \mathcal{F}_m \cdot \phi(n)\right) \mu\left(R_3\right).
		\end{aligned}
		$$
		Next, since $m < \phi_1^{-1}\left(p^{-2 d} n^{-4d(d/s+1)}\right)$, we have
		$$
		\# \mathcal{F}_m \leq Q^m < Q^{\frac{\log c p^{2 d} n^{4d(d/s+1)}}{\tau}} < c_3 n^{\frac{4d(d/s+1) \log Q}{\tau}}. 
		$$
		Applying this upper bound, we complete the proof.
	\end{proof}
	
	We are now prepared to estimate the $\mu\left(B \cap \hat{E}_m \cap \hat{E}_n\right)$ in the case that $m < \phi_1^{-1}\left(p^{-2 d} n^{-4d(d/s+1)}\right)$.

	\begin{lemma} \label{l311} Let $T:[0,1]^d \rightarrow[0,1]^d$ be a piecewise expanding map. Let $B \subset[0,1]^d$ be a ball. Then, there exists a constant $c_4$ such that for all integers $m$ and $n$ satisfying $m < \phi_1^{-1}\left(p^{-2 d} n^{-4d(d/s+1)}\right)$, 
		$$
		\mu\left(B \cap \hat{E}_m \cap \hat{E}_n\right) \leq c_4\left(\mu(B)+\phi(m)^{\frac{1}{2}}\right) \cdot r_{m, 1} \cdots r_{m, d} \cdot r_{n, 1} \cdots r_{n, d}.
		$$
		In particular, if $m \geq-\frac{2 \log c^{-1/2}\mu(B)}{\tau}$, then
		$$
		\mu\left(B \cap \hat{E}_m \cap \hat{E}_n\right) \leq 2 c_4 \mu(B) \cdot r_{m, 1} \cdots r_{m, d} \cdot r_{n, 1} \cdots r_{n, d}.
		$$
	\end{lemma}

	\begin{proof} 
		Denote the radius of $B$ by $r_0$. Partition $B$ into $r_0^d \cdot \phi(n)^{-\frac{1}{4}}$ balls with radius
		$r_1:=\phi(n)^{\frac{1}{4d}}$. The collection of these balls is denoted by
		$$
		\left\{B\left(\mathbf{x}_i, r_1\right): 1 \leq i \leq r_0^d \cdot \phi(n)^{-\frac{1}{4}}\right\}.
		$$
		Next, define $u_m=\max _{1 \leq i \leq d}\left(pr_1 / r_{m, i}\right)$ and $u_n=\max _{1 \leq i \leq d}\left(pr_1 / r_{n, i}\right)$. For each ball $B\left(\mathbf{x}_i, r_1\right)$, $B\left(\mathbf{x}_i, r_1\right) \subset f^{-1} B\left(f\left(\mathbf{x}_i\right), p r_1\right)$, and by Lemma \ref{l35}, we have
		$$
		\begin{aligned}
			&B\left(\mathbf{x}_i, r_1\right) \cap \hat{E}_m \cap \hat{E}_n \\ & \subset B\left(\mathbf{x}_i, r_1\right) \cap T^{-m} R\left(f(\mathbf{x}_i), \xi_m\left(\mathbf{x}_i\right)+2 u_m \mathbf{r}_m\right) \cap T^{-n} R\left(f(\mathbf{x}_i), \xi_n\left(\mathbf{x}_i\right)+2 u_n \mathbf{r}_n\right).
		\end{aligned}
		$$
		Hence,
		$$
		B \cap \hat{E}_m \cap \hat{E}_n \subset \bigcup_{i \leq r_0^d \cdot \phi(n)^{-\frac{1}{4}}} G_i \cap T^{-n} R\left(f(\mathbf{x}_i), \xi_n\left(\mathbf{x}_i\right)+2 u_n \mathbf{r}_n\right),
		$$
		where
		$$
		G_i=B\left(\mathbf{x}_i, r_1\right) \cap T^{-m} R\left(f(\mathbf{x}_i), \xi_m\left(\mathbf{x}_i\right)+2 u_m \mathbf{r}_m\right).
		$$
		By Lemmas \ref{l36} and \ref{l310}, we have
		$$
		\begin{aligned}
			\mu\left(B \cap \hat{E}_m \cap \hat{E}_n\right) & \leq \sum_{i \leq r_0^d \cdot \phi(n)^{-\frac{1}{4}}} \mu\left(G_i \cap T^{-n} R\left(f(\mathbf{x}_i), \xi_n\left(\mathbf{x}_i\right)+2 u_n \mathbf{r}_n\right)\right) \\
			& \leq \sum_{i \leq r_0^d \cdot \phi(n)^{-\frac{1}{4}}}\left(\mu\left(G_i\right)+c_3 \phi(n)^{\frac{1}{2}}\right) \mu\left(R\left(f(\mathbf{x}_i), \xi_n\left(\mathbf{x}_i\right)+2 u_n \mathbf{r}_n\right)\right) \\
			& \leq \sum_{i \leq r_0^d \cdot \phi(n)^{-\frac{1}{4}}}\left(\mu\left(G_i\right)+c_3 \phi(n)^{\frac{1}{2}}\right)\left(r_{n, 1} \ldots r_{n, d}+c_1\left(2 u_n\right)^s\right) \\
			& =\left(\sum_{i \leq r_0^d \cdot \phi(n)^{-\frac{1}{4}}} \mu\left(G_i\right)+c_3 r_0^d \phi(n)^{\frac{1}{4}}\right)\left(r_{n, 1} \ldots r_{n, d}+c_1\left(2 u_n\right)^s\right).
		\end{aligned}
		$$
		Next, we estimate $\sum_{i \leq r_0^d \cdot \phi(n)^{-\frac{1}{4}}} \mu\left(G_i\right)$. Partition $B$ into $r_0^d \phi(m)^{-\frac{1}{2}}$ balls with radius $r_2:=\phi(m)^{\frac{1}{2d}}$. Denote the collection of these balls by
		$$\left\{B\left(\mathbf{z}_j, r_2\right): 1 \leq j \leq r_0^d \phi(m)^{-\frac{1}{2}}\right\}.$$
		Define $v_m=\max _{1 \leq i \leq d}\left(pr_2 / r_{m, i}\right).$ By the condition of $m < \phi_1^{-1}\left(p^{-2 d} n^{-4d(d/s+1)}\right)$, then $v_m \geq u_m$. For any $\mathbf{x} \in$ $B\left(\mathbf{z}_j, r_2\right) \cap G_i$, we have
		$$
		T^m \mathbf{x} \in R\left(f(\mathbf{x}_i), \xi_m\left(\mathbf{x}_i\right)+2 u_m \mathbf{r}_m\right) \quad \text { and } \quad \mathbf{x}_i \in B\left(\mathbf{z}_j, r_1+r_2\right).
		$$
		By the $p$-Lipschitz condition,
		$$
		f(\mathbf{x}_i) \in B\left(f(\mathbf{z}_j), pr_1+pr_2\right)\subset R\left(f(\mathbf{z}_j),\left(u_m+v_m\right) \mathbf{r}_m\right).
		$$
		Hence,
		$$
		\begin{aligned}
			T^m \mathbf{x} \in R\left(f(\mathbf{x}_i), \xi_m\left(\mathbf{x}_i\right)+2 u_m \mathbf{r}_m\right) & \subset R\left(f(\mathbf{z}_j), \xi_m\left(\mathbf{x}_i\right)+3 u_m \mathbf{r}_m+v_m \mathbf{r}_m\right) \\
			& \subset R\left(f(\mathbf{z}_j), \xi_m\left(\mathbf{x}_i\right)+4 v_m \mathbf{r}_m\right).
		\end{aligned}
		$$
		By the triangle inequality, 
		$$
		R\left(f(\mathbf{z}_j), \xi_m\left(\mathbf{x}_i\right)-2 v_m \mathbf{r}_m\right) \subset R\left(f(\mathbf{z}_j), \xi_m\left(\mathbf{x}_i\right)-\left(u_m+v_m\right) \mathbf{r}_m\right) \subset R\left(f(\mathbf{x}_i), \xi_m\left(\mathbf{x}_i\right)\right).
		$$
		Therefore,
		$$
		\mu\left(R\left(f(\mathbf{z}_j), \xi_m\left(\mathbf{x}_i\right)-2 v_m \mathbf{r}_m\right)\right) \leq \mu\left(R\left(f(\mathbf{x}_i), \xi_m\left(\mathbf{x}_i\right)\right)\right)=r_{m, 1} \cdots r_{m, d},
		$$
		which indicates that
		$$
		l_m\left(\mathbf{z}_j\right) \geq l_m\left(\mathbf{x}_i\right)-2 v_m.
		$$
		Using this, we can obtain that
		$$
		T^m \mathbf{x} \in R\left(f(\mathbf{z}_j), \xi_m\left(\mathbf{x}_i\right)+4 v_m \mathbf{r}_m\right) \subset R\left(f(\mathbf{z}_j), \xi_m\left(\mathbf{z}_j\right)+6 v_m \mathbf{r}_m\right)
		,$$
		and further 
		$$
		\begin{aligned}
			&	B\left(\mathbf{z}_j, r_2\right) \cap B\left(\mathbf{x}_i, r_1\right) \cap T^{-m} R\left(f(\mathbf{x}_i), \xi_m\left(\mathbf{x}_i\right)+2 u_m \mathbf{r}_m\right)  \\
			& \subset B\left(\mathbf{z}_j, r_2\right) \cap T^{-m} R\left(f(\mathbf{z}_j), \xi_m\left(\mathbf{z}_j\right)+6 v_m \mathbf{r}_m\right).
		\end{aligned}
		$$
		Therefore, 
		$$
		\begin{aligned}
			& \bigcup_{i \leq r_0^d \cdot \phi(n)^{-\frac{1}{4}}} B\left(\mathbf{x}_i, r_1\right) \cap T^{-m} R\left(f\left(\mathbf{x}_i\right), \xi_m\left(\mathbf{x}_i\right)+2 u_m \mathbf{r}_m\right) \\
			& \quad \subset \bigcup_{i \leq r_0^d \cdot \phi(m)^{-\frac{1}{2}}}
			B\left(\mathbf{z}_j, r_2\right) \cap T^{-m} R\left(f\left(\mathbf{z}_j\right), \xi_m\left(\mathbf{z}_j\right)+6 v_m \mathbf{r}_m\right).
		\end{aligned}
		$$
		By the inclusion above and noting that the set is a disjoint union, we have
		$$
		\begin{aligned}
			& \sum_{i \leq r_0^d \cdot \phi(n)^{-\frac{1}{4}}} \mu\left(B\left(\mathbf{x}_i, r_1\right) \cap T^{-m} R\left(f(\mathbf{x}_i), \xi_m\left(\mathbf{x}_i\right)+2 u_m \mathbf{r}_m\right)\right) \\
			= & \mu\left(\bigcup_{i \leq r_0^d \cdot \phi(n)^{-\frac{1}{4}}} B\left(\mathbf{x}_i, r_1\right) \cap T^{-m} R\left(f(\mathbf{x}_i), \xi_m\left(\mathbf{x}_i\right)+2 u_m \mathbf{r}_m\right)\right) \\
			\leq & \mu\left(\bigcup_{j \leq r_0^d \cdot \phi(m)^{-\frac{1}{2}}} B\left(\mathbf{z}_j, r_2\right) \cap T^{-m} R\left(f(\mathbf{z}_j), \xi_m\left(\mathbf{z}_j\right)+6 v_m \mathbf{r}_m\right)\right) \\
			\leq & \sum_{j \leq r_0^d \cdot \phi(m)^{-\frac{1}{2}}} \mu\left(B\left(\mathbf{z}_j, r_2\right) \cap T^{-m} R\left(f(\mathbf{z}_j), \xi_m\left(\mathbf{z}_j\right)+6 v_m \mathbf{r}_m\right)\right).
		\end{aligned} 
		$$
		By Lemma \ref{l36}, 
		$$
		\begin{aligned}
			& \sum_{j \leq  r_0^d \cdot \phi(m)^{-\frac{1}{2}}} \mu\left(B\left(\mathbf{z}_j, r_2\right) \cap T^{-m} R\left(f(\mathbf{z}_j), \xi_m\left(\mathbf{z}_j\right)+6 v_m \mathbf{r}_m\right)\right) \\
			\leq & \sum_{j \leq r_0^d \cdot \phi(m)^{-\frac{1}{2}}}\left(\mu\left(B\left(\mathbf{z}_j, r_2\right)\right)+\phi(m)\right) \mu\left(R\left(f(\mathbf{z}_j), \xi_m\left(\mathbf{z}_j\right)+6 v_m \mathbf{r}_m\right)\right) \\
			\leq & \sum_{j \leq r_0^d \cdot \phi(m)^{-\frac{1}{2}}}\left(\mu\left(B\left(\mathbf{z}_j, r_2\right)\right)+\phi(m)\right)\left(r_{m, 1} \cdots r_{m, d}+c_1\left(6 v_m\right)^s\right) \\
			= & \left(\mu(B)+r_0^d \phi(m)^{\frac{1}{2}}\right)\left(r_{m, 1} \cdots r_{m, d}+c_1\left(6 v_m\right)^s\right).
		\end{aligned}
		$$
		Thus, combining the above estimations, we have
		$$
		\begin{aligned}
			\mu\left(B \cap \hat{E}_m \cap \hat{E}_n\right) \leq & \left(\left(\mu(B)+r_0^d \phi(m)^{\frac{1}{2}}\right)\left(r_{m, 1} \cdots r_{m, d}+c_1\left(6 v_m\right)^s\right)+c_3 r_0^d \phi(n)^{\frac{1}{4}}\right) \\
			& \cdot\left(r_{n, 1} \ldots r_{n, d}+c_1\left(2 u_n\right)^s\right).
		\end{aligned}
		$$
		Recall that $r_{m, i}>m^{-2}$ and $r_{n, i}>n^{-2}$. It follows from $r_1=\phi(n)^{\frac{1}{4 d}}$ and $r_2=\phi(m)^{\frac{1}{2 d}}$ that 
		$$
		v_m=\max _{1 \leq i \leq d} \frac{pr_2}{r_{m, i}} \leq pm^{2} \phi(m)^{\frac{1}{2 d}} \quad \text { and } \quad u_n=\max _{1 \leq i \leq d} \frac{pr_1}{r_{n, i}} \leq pn^{2} \phi(n)^{\frac{1}{4 d}}.
		$$
		Hence, there exists a constant $c_4$ such that
		$$
		\mu\left(B \cap \hat{E}_m \cap \hat{E}_n\right) \leq c_4\left(\mu(B)+\phi(m)^{\frac{1}{2}}\right) \cdot r_{m, 1} \cdots r_{m, d} \cdot r_{n, 1} \cdots r_{n, d}.
		$$
	\end{proof}

	\begin{lemma} [{\cite[Lemma 2.12]{7}}] \label{l312} Let $\nu$ be a Borel probability measure on $\mathbb{R}^d$ and let $E$ be a measurable subset of $\mathbb{R}^d$. Suppose there exist constants $r_0$ and $\alpha>0$ such that for any ball $B:=B(\mathbf{x}, r)$ with $r<r_0$, we have $\nu(B \cap E) \geq \alpha \nu(B)$. Then,
		$$
		\nu(E)=1 .
		$$
	\end{lemma}
	They proved that the Lebesgue density theorem holds for all Borel probability measures on $\mathbb{R}^d$.
	
	\begin{lemma}[Paley-Zygmund Inequality] \label{l313} Let $(X, \mathcal{B}, \mu)$ be a probability space and $Z$ be a non-negative random variable. Then for any $0<\lambda<1$,  
		$$
		\mu(\{x \in X: Z(x)>\lambda \mathbb{E}(Z)\}) \geq(1-\lambda)^2 \frac{\mathbb{E}(Z)^2}{\mathbb{E}\left(Z^2\right)}.
		$$
	\end{lemma}
	
	\begin{lemma} \label{l314} There exists a constant $\alpha_1>0$ such that for any ball $B \subset[0,1]^d$, 
		$$
		\mu\left(B \cap \hat{\mathcal{R}}^f\left(\left\{\mathbf{r}_n\right\}\right)\right) \geq \alpha_1 \mu(B).
		$$
		In particular,  $\mu\left(\hat{\mathcal{R}}^f\left(\left\{\mathbf{r}_n\right\}\right)\right)=1$.
	\end{lemma}
	
	\begin{proof}
		Fix a ball $B \subset[0,1]^d$ with $\mu(B)>0$. We can replace at most finitely many $\mathbf{r}_n$ 's by $\mathbf{0}$, so that the estimations given in Lemmas \ref{l37}, \ref{l39} and \ref{l311} hold for all $n \in \mathbb{N}$. This manipulation does not affect the property $\sum_{n \geq 1} r_{n, 1} \cdots r_{n, d}=\infty$, and the resulting set is smaller than the original one.
		
		Let $N \in \mathbb{N}$, define $Z_N(x)=\sum_{n=1}^N \chi_{B \cap \hat{E}_n}(x)$. By Lemma \ref{l37}, one has
		$$
		\frac{1}{2} \mu(B) \sum_{n=1}^N r_{n, 1} \cdots r_{n, d} \leq \mathbb{E}\left(Z_N\right)=\sum_{n=1}^N \mu\left(B \cap \hat{E}_n\right) \leq 2 \mu(B) \sum_{n=1}^N r_{n, 1} \cdots r_{n, d}.
		$$
		On the other hand,
		$$
		\begin{aligned}
			\mathbb{E}\left(Z_N^2\right) & =\sum_{m, n=1}^N \mu\left(B \cap \hat{E}_m \cap \hat{E}_n\right) 
			=2 \sum_{n=1}^N \sum_{m=1}^{n-1} \mu\left(B \cap \hat{E}_m \cap \hat{E}_n\right)+\sum_{n=1}^N \mu\left(B \cap \hat{E}_n\right).
		\end{aligned}
		$$
		Using the estimations from Lemmas \ref{l39} and \ref{l311}, we obtain
		$$
		\begin{aligned}
			&\sum_{n=1}^N \sum_{m=1}^{n-1} \mu\left(B \cap \hat{E}_m \cap \hat{E}_n\right) \\ & =  \sum_{n=1}^N \sum_{m=\left\lceil\phi_1^{-1}\left(p^{-2 d} n^{-4d(d/s+1)}\right)\right\rceil}^{n-1} \mu\left(B \cap \hat{E}_m \cap \hat{E}_n\right)\\
			& +\sum_{n=1}^N \sum_{m=1}^{\left\lfloor\phi_1^{-1}\left(p^{-2 d} n^{-4d(d/s+1)}\right)\right\rfloor} \mu\left(B \cap \hat{E}_m \cap \hat{E}_n\right) \\
			\leq & \sum_{n=1}^N \sum_{m=\left\lceil\phi_1^{-1}\left(p^{-2 d} n^{-4d(d/s+1)}\right)\right\rceil}^{n-1} c_2 \mu(B) \cdot r_{n, 1} \cdots r_{n, d} \cdot\left(r_{m, 1} \cdots r_{m, d}+\phi(n-m)\right) \\
			& +\sum_{n=1}^N \sum_{m=1}^{\left\lfloor\phi_1^{-1}\left(p^{-2 d} n^{-4d(d/s+1)}\right)\right\rfloor} 2 c_4 \mu(B) \cdot r_{m, 1} \cdots r_{m, d} \cdot r_{n, 1} \cdots r_{n, d} ,
		\end{aligned}
		$$
		where $\lceil\cdot\rceil$ and $\lfloor\cdot\rfloor$  denote the ceiling and floor function, respectively.
		Hence, there exist $c_5$ and $c_6$ such that
		$$
		\begin{aligned}
			& \sum_{n=1}^N \sum_{m=1}^{n-1} \mu\left(B \cap \hat{E}_m \cap \hat{E}_n\right) \\
			\leq & c_5 \mu(B) \sum_{n=1}^N \sum_{m=1}^{n-1} r_{m, 1} \cdots r_{m, d} \cdot r_{n, 1} \cdots r_{n, d}+c_6 \mu(B) \sum_{n=1}^N r_{n, 1} \cdots r_{n, d}.
		\end{aligned}
		$$
		Substituting the estimations for $\mathbb{E}\left(Z_N\right)$ and $\mathbb{E}\left(Z_N^2\right)$ derived above, we have
		$$
		\begin{aligned}
			&	\mu\left(Z_N>\lambda \mathbb{E}\left(Z_N\right)\right)  \geq(1-\lambda)^2 \frac{\mathbb{E}\left(Z_N\right)^2}{\mathbb{E}\left(Z_N^2\right)} \\
			& \geq(1-\lambda) \frac{\left(\frac{\mu(B)}{2} \sum_{n=1}^N r_{n, 1} \cdots r_{n, d}\right)^2}{2 c_5 \mu(B)\left(\sum_{n=1}^N r_{n, 1} \cdots r_{n, d}\right)^2+\left(2+2 c_6\right)\mu(B) \sum_{n=1}^N r_{n, 1} \cdots r_{n, d}}.
		\end{aligned}
		$$
		Taking the limit as $N \rightarrow \infty$, we get
		$$
		\begin{aligned}
			\mu\left(B \cap \hat{\mathcal{R}}^f\left(\left\{\mathbf{r}_n\right\}\right)\right) & =\mu\left(\lim \sup \left(B \cap \hat{E}_n\right)\right) \geq \mu\left(\lim \sup \left(Z_N>\lambda \mathbb{E}\left(Z_N\right)\right)\right) \\
			& \geq \lim \sup \mu\left(Z_N>\lambda \mathbb{E}\left(Z_N\right)\right)\geq(1-\lambda) \frac{\mu(B)}{8 c_5}.
		\end{aligned}
		$$
		Let $\lambda=1 / 2$ and $\alpha_1=\left(16 c_5\right)^{-1}$. This completes the proof of the first part of the lemma. The second part follows immediately from Lemma \ref{l312}.
	\end{proof}
	
	Using Lemmas \ref{l34} and \ref{l314}, we conclude that $$
	\mu\left(\mathcal{R}^f\left(\left\{\mathbf{r}_n\right\}\right)\right)=1.
	$$
	
	Note that for the divergent part, the boundedness condition on the sequence $\left\{\frac{\left|\mathbf{r}_n\right|}{\left|\mathbf{r}_n\right|_{\text {min }}}\right\}$ is not invoked.
	
	\section{Proof of Theorem \ref{13} } 
	Recall that, as in the proof of Theorem \ref{12}, we will focus on proving Theorem \ref{13} under Condition \ref{131} and specify where the Definition \ref{ab} on $f$ is applied. And if the Condition \ref{131} is strengthened to Condition \ref{132}, the proof also incorporates a detailed discussion of the case where the restriction of Definition \ref{ab} on $f$ is removed. 
	
	Recall for any $\mathbf{x} \in[0,1]^d$ and $\delta>0$, $H(f(\mathbf{x}), \delta)$ is defined by
	$$
	H(f(\mathbf{x}), \delta)=f(\mathbf{x})+\left\{\mathbf{z} \in[-1,1]^d: \left|z_1 \cdots z_d\right|<\delta \right\}.
	$$
	Let $\left\{\delta_n\right\}$ be the sequence defined as in Theorem \ref{13}. Then, $\mathcal{R}^{f \times}\left(\delta_n\right)=\lim \sup D_n$, 
	where
	$$
	D_n:=\left\{\mathbf{x} \in[0,1]^d: T^n \mathbf{x} \in H\left(f(\mathbf{x}), \delta_n\right)\right\}.
	$$
	
	$\partial H(F(\mathbf{x}), \delta)$ can be divided into two parts. One part is from the boundary of $[-1,1]^d$, and the other part comes from the hypersurface $\left|z_1 \cdots z_d\right|=\delta$. The set $\left\{\mathbf{z} \in[-1,1]^d:\left|z_1 \cdots z_d\right|=\delta\right\}$ is an $(d-1)$-dimensional hypersurface, defined by a scalar equation, and thus is a subset in $[-1,1]^{d-1}$.
	So we conclude that for any $\mathbf{x} \in[0,1]^d$, there exists a constant $K_1$ such that
	$$
	\sup _{0 \leq \delta \leq 1} M^{*(d-1)}(\partial H(f(\mathbf{x}), \delta)) \leq K_1.
	$$
	
	\begin{lemma} [{\cite[Lemma 4]{16}}] \label{241} Given $d \in \mathbb{N}$ and $\delta>0$, let
		$$
		H_d(\delta):=\left\{\left(x_1, \ldots, x_d\right) \in[0,1)^d:\left\|x_1\right\| \ldots\left\|x_d\right\|<\delta\right\}.
		$$
		Then,
		$$
		m_d\left(H_d(\delta)\right)= \begin{cases}1 & \text { if } \delta \geq 2^{-d}, \\ 2^d \delta\left(\sum_{t=0}^{d-1} \frac{1}{t!}\left(\log \frac{1}{2^d \delta}\right)^t\right) & \text { if } \delta<2^{-d}.\end{cases}
		$$
		Here, $\|\cdot\|$ denotes the distance to the nearest integer.
	\end{lemma}
	
	We adopt several operations from \cite{7}, which are essential for our estimations. These are summarized below and directly applied in our estimations.
	
	By Lemma \ref{241}, for any $0<\delta<r<1$, we have 
	\begin{equation}
		m_d(B(\mathbf{x}, r) \cap H(\mathbf{x}, \delta))=2^d \delta\left(\sum_{t=0}^{d-1} \frac{1}{t!}\left(\log \frac{r^d}{\delta}\right)^t\right).
	\end{equation}
	Next, for any $\delta>0$, it holds that
	\begin{equation} \label{eq:4.2}
		m_d(H(\mathbf{x}, \delta)) \leq d 2^d \delta(-\log \delta)^{d-1}.
	\end{equation}
	For the case that $r^d>\sqrt{\delta}$, we have
	\begin{equation} \label{eq:4.3}
		m_d(B(\mathbf{x}, r) \cap H(\mathbf{x}, \delta)) \geq \frac{2}{(d-1)!} \delta(-\log \delta)^{d-1}.
	\end{equation}
	In particular, if $\delta=\delta_n+c_0 n^{-3}$ for some $c_0 \leq  2^{d+2} p^d$, then for sufficiently large $n$ with $\delta_n>n^{-2}$, 
	\begin{equation} \label{eq:4.4}
		m_d\left(H\left(\mathbf{x}, \delta_n+c_0 n^{-3}\right)\right) \leq d 2^{d+1} \delta_n\left(-\log \delta_n\right)^{d-1}.
	\end{equation}
	
	In this section, we fix the collection $\mathcal{C}_2$ of subsets $E \subset[0,1]^d$ satisfying the bounded property
	$$
	(\mathbf{P 2}): \quad \sup _{E \in \mathcal{C}_2} M^{*(d-1)}(\partial E)<2 dc_M+K_1+\frac{Kc_M}{1-L^{-(d-1)}},
	$$
	where $K_1$ is as defined above, $L$ and $K$ are given in Definition \ref{pem}, and $c_M$ is a constant defined before Definition \ref{pem}. It is evident that all hyperrectangles $R \subset$ $[0,1]^d$ and all hyperboloids $H \subset$ $[0,1]^d$ satisfy the bounded property $(\mathbf{P 2})$.
	
	As in the proof of Lemma \ref{JBXZ1}, we will frequently rely on the following lemma.
	
	\begin{lemma} \label{l42}
		Let $B(f(\mathbf{x}), p r) \subset[0,1]^d$ be a ball centered at $f(\mathbf{x}) \in[0,1]^d$ and radius $pr>0$. Then, for any subset $A$ of $f^{-1} B(f(\mathbf{x}), p r)$, 
		$$
		A \cap D_n \subset A \cap T^{-n} H\left(f(\mathbf{x}), \delta_n+(2 p)^d r\right).
		$$
		Furthermore, if  $(2 p)^d r \leq \delta_n$, 
		$$
		A \cap T^{-n} H\left(f(\mathbf{x}), \delta_n-(2 p)^d r\right) \subset A \cap D_n.
		$$
	\end{lemma}
	
	\begin{proof}
		Let $\mathbf{z} \in A$, we have $f(\mathbf{z}) \in B(f(\mathbf{x}), p r)$, implying there exists $\mathbf{r}=\left(r_1, \ldots, r_d\right) \in \mathbb{R}^d$ with $|\mathbf{r}|<r$ such that $f(\mathbf{z})=f(\mathbf{x})+p \mathbf{r}$. If $\mathbf{z}$ also belongs to $D_n$, then $T^n \mathbf{z} \in H\left(f(\mathbf{z}), \delta_n\right)$. Write $T^n \mathbf{z}=\left(\tilde{z_1}, \ldots, \tilde{z_d}\right)$.
		Then,
		$$
		\left|\tilde{z}_1-f_1\left(\mathbf{z}\right)\right| \cdots\left|\tilde{z}_d-f_d\left(\mathbf{z}\right)\right|<\delta_n,
		$$
		which implies
		$$
		\begin{aligned}
			& \left|\tilde{z}_1-f_1\left(\mathbf{x}\right)\right| \cdots\left|\tilde{z}_d-f_d\left(\mathbf{x}\right)\right| =\left|\tilde{z}_1-f_1\left(\mathbf{z}\right)+p r_1\right| \cdots\left|\tilde{z}_d-f_d\left(\mathbf{z}\right)+p r_d\right| \\
			& \leq\left|\tilde{z}_1-f_1\left(\mathbf{z}\right)\right| \cdots\left|\tilde{z}_d-f_d\left(\mathbf{z}\right)\right|+(2 p)^d r <\delta_n+(2 p)^d r .
		\end{aligned}
		$$
		This shows that $T^n \mathbf{z} \in H\left(f(\mathbf{x}), \delta_n+(2 p)^d r\right)$. Therefore, $$\mathbf{z} \in A \cap T^{-n} H\left(f(\mathbf{x}), \delta_n+(2 p)^d r\right).$$ 
		
		The second inclusion follows similarly.
	\end{proof}

	\subsection{The convergence part.}\label{HC}
	Let $V$ be an open set with $\mu(V)=1$ such that the density $h$ of $\mu$, when restricted on $V$, is bounded both above and below by constants $\mathfrak{c} \geq 1$ and $\mathfrak{c}^{-1}$, respectively. Additionally, since $\mu \circ f^{-1} \ll \mu$, by Lemma \ref{DJ}, we obtain that
	$\exists W=f^{-1}(V)$ with $\mu(W)=1$, which implies that $\mu(V \cap W)=1$. In other words, for hyperboloids centered at $\mathbf{x} \in V \cap W$ or $f(\mathbf{x}) \in V$, we can obtain both upper and lower bounds of their $\mu$-measure. Notably, under the Condition \ref{132} of Theorem \ref{13}, these intermediate steps can be omitted.

	\begin{proposition}
		If $\sum_{n \geq 1} \delta_n\left(-\log \delta_n\right)^{d-1}<\infty$,  $\mu\left(\mathcal{R}^{f \times}\left(\left\{\delta_n\right\}\right)\right)=0$.
	\end{proposition}
	
	\begin{proof}
		Let $n \geq 1$.  Following a similar argument as in the proof of Proposition \ref{PRC}, we can find a collection of pairwise disjoint balls that cover the set $V \cap W$. Denote the collection of these balls by
		$$
		\left\{B\left(\mathbf{x}_i, n^{-2}\right): \mathbf{x}_i \in V \cap W \right\}_{i \in \mathcal{I}_1},
		$$
		with $\sharp \mathcal{I}_1 \leq (n^2 / 2)^{ d}$.
		By the $p$-Lipschitz condition, 
		$$
		B\left(\mathbf{x}_i, n^{-2}\right) \subset f^{-1} B\left(f\left(\mathbf{x}_i\right), p n^{-2}\right),
		$$
		where $f\left(\mathbf{x}_i\right) \in V$.
		Next, by applying Lemma \ref{l42},
		$$
		\begin{aligned}
			V \cap W \cap D_n&\subset \bigcup_{i \leq(n^2 / 2)^{ d}} B\left(\mathbf{x}_i, n^{-2}\right) \cap D_n \\ & \subset \bigcup_{i \leq (n^2 / 2)^{ d}}  B\left(\mathbf{x}_i, n^{-2}\right) \cap T^{-n} H\left(f(\mathbf{x}_i), \delta_n+(2 p)^d n^{-2}\right) .
		\end{aligned}
		$$
		By the polynomial-mixing property of $\mu$,
		$$
		\begin{aligned}
			\mu\left(V \cap W \cap D_n\right) & \leq \sum_{i \leq (n^2 / 2)^{ d}} \mu\left(B\left(\mathbf{x}_i, n^{-2}\right) \cap T^{-n} H\left(f(\mathbf{x}_i), \delta_n+(2 p)^d n^{-2}\right)\right) \\
			& \leq \sum_{i \leq(n^2 / 2)^{ d}}\left(\mu\left(B\left(\mathbf{x}_i, n^{-2}\right)\right)+\phi(n)\right) \mu\left(H\left(f(\mathbf{x}_i), \delta_n+(2 p)^d n^{-2}\right)\right) .
		\end{aligned}
		$$
		Since $\left.h\right|_V \leq \mathfrak{c} $, and by \eqref{eq:4.2}, the above sum is majorized by 
		$$
		\begin{aligned}
			& \sum_{i \leq (n^2 / 2)^{ d}}\left(\mu\left(B\left(\mathbf{x}_i, n^{-2}\right)\right)+\phi(n)\right) \mathfrak{c} \cdot d 2^d\left(\delta_n+(2 p)^d n^{-2}\right)\left(-\log \left(\delta_n+(2 p)^d n^{-2}\right)\right)^{d-1} \\
			\leq & \mathfrak{c} d 2^d\left(\mathfrak{c}+(n^2 / 2)^{ d} \phi(n)\right)\left(\delta_n+(2 p)^d n^{-2}\right)\left(-\log \left(\delta_n+(2 p)^d n^{-2}\right)\right)^{d-1} \\
			\leq & \mathfrak{c} d 2^d\left(\mathfrak{c}+(n^2 / 2)^{ d} \phi(n)\right)\left(\delta_n\left(-\log \delta_n\right)^{d-1}+(2 p)^d n^{-2}\left(-\log \left((2 p)^d n^{-2}\right)\right)^{d-1}\right).
		\end{aligned}
		$$
		Therefore, 
		$$
		\begin{aligned}
			&\sum_{n=1}^{\infty} \mu\left(V \cap W \cap D_n\right) \\ & \leq \sum_{n=1}^{\infty} \mathfrak{c} d 2^d\left(\mathfrak{c}+(n^2 / 2)^{ d} \phi(n)\right)\left(\delta_n\left(-\log \delta_n\right)^{d-1}+(2 p)^d n^{-2}\left(-\log \left((2 p)^d n^{-2}\right)\right)^{d-1}\right).
		\end{aligned}
		$$
		By $\sum_{n \geq 1} \delta_n\left(-\log \delta_n\right)^{d-1}<\infty$ and $\sum_{n \geq 1}  n^{-2}\left(-\log \left( n^{-2}\right)\right)^{d-1}<\infty$, we have 
		$$
		\sum_{n=1}^{\infty} \mu\left(V \cap W \cap D_n\right)<\infty.
		$$
		By Borel-Cantelli lemma and the definition of $V \cap W$,
		$$
		\mu\left(\mathcal{R}^{f \times}\left(\left\{\delta_n\right\}\right)\right)=0.
		$$
		Under the Condition \ref{132} of Theorem \ref{13}, it suffices to replace $V \cap W$ with $[0,1]^d$. Moreover, it is unnecessary to assume that $f$ satisfies Definition \ref{ab}, and the convergence part of Theorem \ref{13} also holds.
	\end{proof}

	\subsection{The divergence part.} \label{HD}
	
	In analogy with Section \ref{RD}, we note that
	$$
	\sum_{n:  \delta_n>n^{-2}} \delta_n\left(-\log \delta_n\right)^{d-1}=\infty \Longleftrightarrow \sum_{n=1}^{\infty} \delta_n\left(-\log \delta_n\right)^{d-1}=\infty
	$$
	and
	\begin{equation}
		\limsup _{n: \delta_n>n^{-2}} D_n \subset \limsup _{n \rightarrow \infty} D_n \text {. }
		\label{Eq.4.5}
	\end{equation}
	We can extend  $\delta_n\left(-\log \delta_n\right)^{d-1}$ such that when $\delta_n=0$, $\delta_n\left(-\log \delta_n\right)^{d-1}=0$. This extension allows us to assume that $\delta_n$ is either greater than $n^{-2}$ or equal to 0. Then we can follow the similar steps as Section \ref{RD} to prove this part. However, we will pursue a more general approach to prove this part; that is, assume $\delta_n>n^{-2}$ only for sufficiently large $n>N_0$ to prove the left side of \eqref{Eq.4.5} is a full $\mu$-measure set.

	Let $V \cap W$ be the open set with $\mu(V \cap W)=1$ as defined in Section \ref{HC}. Then $V \cap W$ can be written as
	$$
	V \cap W=\bigcup_{\delta>0} V_\delta \cap W_\delta,
	$$
	where $V_\delta \cap W_\delta:=\{\mathbf{x} \in V \cap W: B(\mathbf{x}, \delta) \subset V \cap W  \ \text{and} \ B(f(\mathbf{x}), \delta) \subset V  \}$. It is evident that the sets $V_\delta \cap W_\delta$ are monotonic with respect to $\delta$.
	
	The goal of this subsection is to prove the following lemma.
	
	\begin{lemma} \label{l44}
		For any $\delta_1>0$ and $\delta_2>0$, there exists a constant $\alpha_2(\delta_1, \delta_2)>0$ depending on $\delta_1$ and $\delta_2$, such that for all $r<\frac{\min \left\{\delta_1, \delta_2\right\}}{2}$ and $\mathbf{x} \in V_{\delta_1} \cap W_{\delta_2}$,
		$$
		\mu\left(B(\mathbf{x}, r) \cap \mathcal{R}^{f \times}\left(\left\{\delta_n\right\}\right) \right) \geq \alpha_2(\delta_1, \delta_2) \mu(B(\mathbf{x}, r)).
		$$
	\end{lemma}
	Now, we obtain the divergence part of Theorem \ref{13} based on Lemma \ref{l44}. The proof follows the same structure as in \cite{7}, with the only difference being the definition of the full $\mu$-measure set.
	
	\begin{proof}[Proof of the divergence part of Theorem \ref{13} modulo Lemma \ref{l44}]
		For any $\delta_1>0$ and $\delta_2>0$ with $\mu\left(V_{\delta_1} \cap W_{\delta_2}\right)>0$, we claim that
		$$
		\mu\left(V_{\delta_1} \cap W_{\delta_2} \cap \mathcal{R}^{f \times}\left(\left\{\delta_n\right\}\right) \right)=\mu\left(V_{\delta_1} \cap W_{\delta_2}\right).
		$$
		Assume by contradiction that $
		\mu\left(V_{\delta_1} \cap W_{\delta_2} \cap \mathcal{R}^{f \times}\left(\left\{\delta_n\right\}\right) \right)<\mu\left(V_{\delta_1} \cap W_{\delta_2}\right)
		$. Then, the set $A := V_{\delta_1} \cap W_{\delta_2} \backslash \mathcal{R}^{f \times}\left(\left\{\delta_n\right\}\right)$ has positive $\mu$-measure. By the density theorem \cite[Corollary 2.14]{29}, for $\mu-a.e.$ $\mathbf{x} \in A$,
		$$
		\lim _{r \rightarrow 0} \frac{\mu(B(\mathbf{x}, r) \cap A)}{\mu(B(\mathbf{x}, r))}=1.
		$$
		Thus, for any such $\mathbf{x}$, there exists $r<\frac{\min \left\{\delta_1, \delta_2\right\}}{2}$ small enough such that
		$$
		\mu(B(\mathbf{x}, r) \cap A) \geq\left(1-\alpha_2(\delta_1, \delta_2) \right) \mu(B(\mathbf{x}, r)).
		$$
		This implies that
		$$
		\mu\left(B(\mathbf{x}, r) \cap \mathcal{R}^{f \times}\left(\left\{\delta_n\right\}\right)\right)<\alpha_2(\delta_1, \delta_2) \mu(B(\mathbf{x}, r)),
		$$
		which contradicts Lemma \ref{l44}. Since $V \cap W=\bigcup_{\delta>0} V_\delta \cap W_\delta$ and  $\mu(V \cap W)=1$, we arrive at the conclusion.
	\end{proof}
	
	Under the Condition \ref{132} of Theorem \ref{13}, it suffices to replace $V \cap W$ with $(0,1)^d$. Furthermore, it is unnecessary to assume that $f$ satisfies Definition \ref{ab}, and the above proof also holds. Lemma \ref{l44} is proved by using standard techniques, including the local Lebesgue density theorem and the Paley-Zygmund inequality.

	\begin{lemma} \label{l45}
		Under the conditions of Theorem \ref{13}, let $B$ be a ball centered at $\mathbf{z} \in V_{\delta_1} \cap W_{\delta_2}$ with radius $r_0<\frac{\min \left\{\delta_1, \delta_2\right\}}{2}$. Then, there exists a sufficiently large integer $ N_1>N_0$, for all $n>N_1$,
		$$
		\frac{\mu(B)}{2 \mathfrak{c}  (d-1)!} \cdot \delta_n\left(-\log \delta_n\right)^{d-1} \leq \mu\left(B \cap D_n\right) \leq \mathfrak{c} d 2^{d+2} \mu(B) \cdot \delta_n\left(-\log \delta_n\right)^{d-1}.
		$$
	\end{lemma}

	\begin{proof}
		Let $\delta=\min \left\{\delta_1, \delta_2\right\}$. Since the radius of $B$ is less than $\delta / 2$, by definition of $V_{\delta_1} \cap W_{\delta_2}$, we have $B(\mathbf{x}, \delta / 2) \subset V_{\delta_1} \cap W_{\delta_2}$ for all $\mathbf{x} \in B$.
		
		Partition $B$ into $r_0^d n^{3 d}$ balls with radius $r:=n^{-3}$. The collection of these pairwise disjointed balls is denoted by $\left\{B\left(\mathbf{x}_i, n^{-3}\right): 1 \leq i \leq r_0^d n^{3 d}\right\}$. By Lemma \ref{l42}, for all $n>N_0$,
		$$
		B \cap D_n \supset \bigcup_{i \leq r_0^d n^{3 d}} B\left(\mathbf{x}_i, n^{-3}\right) \cap T^{-n} H\left(f(\mathbf{x}_i), \delta_n-(2p)^d n^{-3}\right) \text {. }
		$$
		Since $\mathbf{x}_i \in V_{\delta_1} \cap W_{\delta_2}$ for all $1 \leq i \leq r_0^d n^{3 d}$, it follows that  $f(\mathbf{x}_i) \in V_{\delta_2}$ for all $1 \leq i \leq r_0^d n^{3 d}$. Note that the density $h$, when restricted on $V$, is bounded from below by $\mathfrak{c}^{-1}$.	Using the polynomial-mixing property of $\mu$, and noting that $B\left(\mathbf{x}_i, n^{-3}\right) \subset f^{-1} B\left(f\left(\mathbf{x}_i\right), p n^{-3}\right)$, for all $n>N_0$,
		$$
		\begin{aligned}
			\mu\left(B \cap D_n\right) & \geq \sum_{i \leq r_0^d n^{3 d}} \mu\left(B\left(\mathbf{x}_i, n^{-3}\right) \cap T^{-n} H\left(f(\mathbf{x}_i), \delta_n-(2p)^d n^{-3}\right)\right) \\
			& \geq \sum_{i \leq r_0^d n^{3 d}}\left(\mu\left(B\left(\mathbf{x}_i, n^{-3}\right)\right)-\phi(n)\right) \mu\left(H\left(f(\mathbf{x}_i), \delta_n-(2p)^d n^{-3}\right)\right) \\
			& \geq \sum_{i \leq r_0^d n^{3 d}}\left(\mu\left(B\left(\mathbf{x}_i, n^{-3}\right)\right)-\phi(n)\right) \cdot \mathfrak{c}^{-1} m_d\left(H\left(f(\mathbf{x}_i), \delta_n-(2p)^d n^{-3}\right)\right)\\
			& =\left(\mu(B)-r_0^d n^{3 d} \phi(n)\right) \cdot \mathfrak{c}^{-1} m_d\left(H\left(f(\mathbf{x}_i), \delta_n-(2p)^d n^{-3}\right)\right).
		\end{aligned}
		$$
		By \eqref{eq:4.3}, for any $\mathbf{x} \in B$ and $n>N_0$ such that $\delta_n<(\delta / 2)^{2d}$,
		$$
		\begin{aligned}
			m_d\left(H\left(f(\mathbf{x}), \delta_n-(2p)^d n^{-3}\right)\right)  & \geq m_d(B(f(\mathbf{x}), \delta / 2) \cap H(f(\mathbf{x}), \delta_n-(2p)^d n^{-3}) \\
			& \geq
			\frac{2}{(d-1)!} (\delta_n-(2p)^d n^{-3})(-\log \delta_n-(2p)^d n^{-3})^{d-1}.
		\end{aligned}
		$$
		Subsequently for all $n>N_0$, 
		$$
		\begin{aligned}
			&\mu\left(B \cap D_n\right) \\ & \geq \left(\mu(B)-r_0^d n^{3 d} \phi(n)\right) \cdot \mathfrak{c}^{-1} \frac{2}{(d-1)!}\left(\delta_n-(2 p)^d n^{-3}\right)\left(-\log \left(\delta_n-(2 p)^d n^{-3}\right)\right)^{d-1}.
		\end{aligned}
		$$
		Futhermore, since $\delta_n > n^{-2}$ for $n>N_0$ and $\mu$ is polynomial-mixing with respect to $(T, \mathcal{C})$, there exists a sufficiently large integer $N_1>N_0$,  for  $n>N_1$, 
		$$
		r_0^d n^{3 d} \phi(n) \leq \mu(B) / 2, \quad \delta_n-(2p)^d n^{-3} \geq \delta_n / 2 \quad \text { and } \quad-\log \left(\delta_n-(2p)^d n^{-3}\right) \geq -\log \delta_n.
		$$
		Thus, for all large $n>N_1$,
		$$
		\mu\left(B \cap D_n\right) \geq \frac{\mu(B)}{2 \mathfrak{c}  (d-1)!} \cdot \delta_n\left(-\log \delta_n\right)^{d-1}.
		$$
		The second inequality follows similarly, using the inclusion
		$$
		B \cap D_n \subset \bigcup_{i \leq r_0^d n^{3 d}} B\left(\mathbf{x}_i, n^{-3}\right) \cap T^{-n} H\left(f(\mathbf{x}_i), \delta_n+(2p)^d n^{-3}\right).
		$$
		This completes the proof.
	\end{proof}
	
	\begin{remark}
		The openness of the set $V \cap W $ is crucial to our argument, as it guarantees a uniform lower bound for the $\mu$-measure of hyperboloids centered at $V_{\delta_1} \cap W_{\delta_2}$ or $V_{\delta_2}$.
	\end{remark}
	Next, we proceed to estimate the $\mu$-measure of the intersection $B \cap D_m \cap D_n$ with $m<n$.
	As stated in the assumptions at the beginning of Section \ref{HD}, we have $\delta_m>m^{-2}$ and $\delta_n>n^{-2}$ for sufficiently large $m>N_0$ and $n$.
	
	\begin{lemma} \label{l46}
		Let $B$ be a ball centered at $\mathbf{z} \in V_{\delta_1} \cap W_{\delta_2}$ and radius $r_0<\frac{\min \left\{\delta_1, \delta_2\right\}}{2}$. There exists a constant $\tilde{c}_1$ and a sufficiently large integer $N_2>e^{\frac{\tau N_0-\ln c}{4 d}}$ such that for sufficiently large integers $m$ and $n$ with $\phi_1^{-1}\left(n^{-4 d}\right) \leq m<n$ and $n>N_2$,
		$$
		\mu\left(B \cap D_m \cap D_n\right) \leq \tilde{c}_1 \mu(B)\left(\delta_m\left(-\log \delta_m\right)^{d-1}+\phi(n-m)\right) \cdot \delta_n\left(-\log \delta_n\right)^{d-1}.
		$$
	\end{lemma}
	
	\begin{proof}
		Let $r_0$ be the radius of $B$. Partition $B$ into $r_0^d n^{3 d}$ balls with radius $r:=n^{-3}$. Let $\left\{B\left(\mathbf{x}_i, n^{-3}\right): 1 \leq i \leq r_0^d n^{3 d}\right\}$ denote the collection of these pairwise disjointed balls. For each ball $B\left(\mathbf{x}_i, n^{-3}\right)$, $B\left(\mathbf{x}_i, n^{-3}\right) \subset f^{-1} B\left(f\left(\mathbf{x}_i\right), p n^{-3}\right)$. By Lemma \ref{l42}, we have
		$$
		\begin{aligned}
			&	B \cap D_m \cap D_n \\ & \subset \bigcup_{i \leq r_0^d n^{3 d}} B\left(\mathbf{x}_i, n^{-3}\right) \cap T^{-m} H\left(f\left(\mathbf{x}_i\right), \delta_m+(2 p)^d n^{-3}\right) \cap T^{-n} H\left(f\left(\mathbf{x}_i\right), \delta_n+(2 p)^d n^{-3}\right).
		\end{aligned}
		$$
		Using a similar argument as in Lemma \ref{l39} and applying the bound in \eqref{eq:4.4}, then for sufficiently large $m>N_0$ such that $\delta_m>m^{-2}$,
		$$
		\begin{aligned}
			& \mu\left(B \cap D_m \cap D_n\right) \\
			\leq & \sum_{i \leq r_0^d n^{3 d}}\left(\mu\left(B\left(\mathbf{x}_i, n^{-3}\right)\right)+\phi(m)\right)\left(\mu\left(H\left(f\left(\mathbf{x}_i\right), \delta_m+(2 p)^d n^{-3}\right)\right)+\phi(n-m)\right) \\
			& \cdot \mu\left(H\left(f\left(\mathbf{x}_i\right), \delta_n+(2 p)^d n^{-3}\right)\right) \\
			\leq & \left(\mu(B)+r_0^d n^{3 d} \phi(m)\right)\left(\mathfrak{c} d 2^{d+1} \delta_m\left(-\log \delta_m\right)^{d-1}+\phi(n-m)\right) \cdot \mathfrak{c} d 2^{d+1} \delta_n\left(-\log \delta_n\right)^{d-1}.
		\end{aligned}
		$$
		Since $\phi_1^{-1}\left(n^{-4 d}\right) \leq m$, we have
		$$
		n^{3 d} \phi(m) \leq n^{3 d} \cdot n^{-4 d}=n^{-d}.
		$$
		Therefore, there exists a sufficiently large $N_2>e^{\frac{\tau N_0-\ln c}{4 d}}$ such that for $n>N_2$,
		$$
		r_0^d n^{3 d} \phi(m)<\mu(B).
		$$
		Hence, there exists a constant $\tilde{c}_1$ such that for sufficiently large integers $m$ and $n$ with $\phi_1^{-1}\left(n^{-4 d}\right) \leq m<n$, and for $n>N_2$,
		$$
		\mu\left(B \cap D_m \cap D_n\right) \leq \tilde{c}_1 \mu(B)\left(\delta_m\left(-\log \delta_m\right)^{d-1}+\phi(n-m)\right) \cdot \delta_n\left(-\log \delta_n\right)^{d-1}.
		$$
	\end{proof}
	
	According to \cite[Lemma 3.5]{7}, the sets $J_n \cap B \cap T^{-n} H(f\left(\mathbf{x}\right), \delta)$ satisfy the bound property $(\mathbf{P2})$, where $J_n \in \mathcal{F}_n$ and $B$ is a ball. As a result, we obtain the following estimation, which can be proved in a similar method to Lemma \ref{l310}.
	
	\begin{lemma} \label{l47}
		Let $T:[0,1]^d \rightarrow[0,1]^d$ be a piecewise expanding map. There exists a constant $\tilde{c}_2$ such that for any integers $m < \phi_1^{-1}\left(  n^{-4d} \right)$, any ball $B$, $\mathbf{x} \in[0,1]^d$ and $0<\delta_1, \delta_2<1$, 
		$$
		\begin{aligned}
			&\mu\left(B \cap T^{-m} H\left(f\left(\mathbf{x}\right), \delta_1\right) \cap T^{-n} H\left(f\left(\mathbf{x}\right), \delta_2\right)\right) \\ & \leq \left(\mu\left(B \cap T^{-m} H\left(f\left(\mathbf{x}\right), \delta_1\right)\right)+\tilde{c}_2 \phi(n)^{\frac{1}{2}}\right) \mu\left(H\left(f\left(\mathbf{x}\right), \delta_2\right)\right).
		\end{aligned}
		$$
	\end{lemma}
	
	We are now prepared to estimate $\mu\left(B \cap D_m \cap D_n\right)$ for the case  $m < \phi_1^{-1}\left(  n^{-4d} \right)$.
	
	\begin{lemma} \label{l48}
		Let $B$ be a ball centered at $\mathbf{z} \in V_{\delta_1} \cap W_{\delta_2}$ and radius $r_0<\frac{\min \left\{\delta_1, \delta_2\right\}}{2}$.
		There exist a constant $\tilde{c}_3$, sufficiently large integers $N_3>N_0$ and $N_4>e^{\frac{\tau N_3-\ln c}{4 d}}$ such that for sufficiently large integers $m$ and $n$ with $N_3<m < \phi_1^{-1}\left(n^{-4 d}\right)$ and $n>N_4$,
		$$
		\mu\left(B \cap D_m \cap D_n\right) \leq \tilde{c}_3 \mu(B) \cdot \delta_m\left(-\log \delta_m\right)^{d-1} \cdot \delta_n\left(-\log \delta_n\right)^{d-1}.
		$$
	\end{lemma}

	\begin{proof}
		We follow the notations of Lemma \ref{l46}. By Lemma \ref{l47},
		$$
		\begin{aligned}
			& \mu\left(B \cap D_m \cap D_n\right) \\
			\leq & \sum_{i \leq r_0^d n^{3 d}} \mu\left(B\left(\mathbf{x}_i, n^{-3}\right) \cap T^{-m} H\left(f\left(\mathbf{x}_i\right), \delta_m+(2 p)^d n^{-3}\right) \cap T^{-n} H\left(f\left(\mathbf{x}_i\right), \delta_n+(2 p)^d n^{-3}\right)\right) \\
			\leq & \sum_{i \leq r_0^d n^{3 d}}\left(\mu\left(B\left(\mathbf{x}_i, n^{-3}\right) \cap T^{-m} H\left(f\left(\mathbf{x}_i\right), \delta_m+(2 p)^d n^{-3}\right)\right)+\tilde{c}_2 \phi(n)^{\frac{1}{2}}\right) \\& \cdot \mu\left(H\left(f\left(\mathbf{x}_i\right), \delta_n+(2 p)^d n^{-3}\right)\right) \\
			\leq & \left(\sum_{i \leq r_0^d n^{3 d}} \mu\left(B\left(\mathbf{x}_i, n^{-3}\right) \cap T^{-m} H\left(f\left(\mathbf{x}_i\right), \delta_m+(2 p)^d n^{-3}\right)\right)+\tilde{c}_2 r_0^d n^{3 d} \phi(n)^{\frac{1}{2}}\right) \\
			& \cdot \mu\left(H\left(f\left(\mathbf{x}_i\right), \delta_n+(2 p)^d n^{-3}\right)\right) .
		\end{aligned}
		$$
		Next, by \eqref{eq:4.4}, for $n>N_0$,
		$$
		\mu\left(H\left(f\left(\mathbf{x}_i\right), \delta_n+(2 p)^d n^{-3}\right)\right) \leq \mathfrak{c} d 2^{d+1} \delta_n\left(-\log \delta_n\right)^{d-1}.
		$$
		We now focus on estimating $\sum_{i \leq r_0^d n^{3 d}} \mu\left(B\left(\mathbf{x}_i, n^{-3}\right) \cap T^{-m} H\left(f\left(\mathbf{x}_i\right), \delta_m+(2 p)^d n^{-3}\right)\right)$.
		
		Partition $B$ into $r_0^d m^{3 d}$ balls with radius $m^{-3}$. Denote the collection of these pairwise disjointed balls as
		$$
		\left\{B\left(\mathbf{z}_j, m^{-3}\right): 1 \leq j \leq r_0^d m^{3 d}\right\}.
		$$
		For any $\mathbf{x} \in B\left(\mathbf{z}_j, m^{-3}\right) \cap B\left(\mathbf{x}_i, n^{-3}\right) \cap T^{-m} H\left(f(\mathbf{x}_i), \delta_m+(2p)^d n^{-3}\right)$, we have
		$$
		T^m \mathbf{x} \in H\left(f\left(\mathbf{x}_i\right), \delta_m+(2 p)^d n^{-3}\right)
		$$
		and
		$$
		\mathbf{x}_i \in B\left(\mathbf{z}_j, m^{-3}+n^{-3}\right) \subset f^{-1} B\left(f\left(\mathbf{z}_j\right), p m^{-3}+p n^{-3}\right) .
		$$
		By applying the same reason from Lemma \ref{l42}, we have
		$$
		T^m \mathbf{x} \in  H\left(f(\mathbf{z}_j), \delta_m+(2p)^d m^{-3}+2 \cdot (2p)^d n^{-3}\right) \subset H\left(f(\mathbf{z}_j),\delta_m+ 2^{d+2}p^d m^{-3}\right).
		$$
		Hence, 
		$$
		\begin{aligned}
			& \bigcup_{i \leq r_0^d n^{3 d}} B\left(\mathbf{x}_i, n^{-3}\right) \cap T^{-m} H\left(f\left(\mathbf{x}_i\right), \delta_m+(2 p)^d n^{-3}\right) \\
			& \quad \subset \bigcup_{j \leq r_0^d m^{3 d}} B\left(\mathbf{z}_j, m^{-3}\right) \cap T^{-m} H\left(f\left(\mathbf{z}_j\right), \delta_m+2^{d+2} p^d m^{-3}\right).
		\end{aligned}
		$$
		Since these balls are pairwise disjointed, we have
		$$
		\begin{aligned}
			& \sum_{i \leq r_0^d n^{3 d}} \mu\left(B\left(\mathbf{x}_i, n^{-3}\right) \cap T^{-m} H\left(f\left(\mathbf{x}_i\right), \delta_m+(2 p)^d n^{-3}\right)\right) \\
			= & \mu\left(\bigcup_{i \leq r_0^d n^{3 d}} B\left(\mathbf{x}_i, n^{-3}\right) \cap T^{-m} H\left(f\left(\mathbf{x}_i\right), \delta_m+(2 p)^d n^{-3}\right)\right) \\
			\leq & \sum_{j \leq r_0^d m^{3 d}} \mu\left(B\left(\mathbf{z}_j, m^{-3}\right) \cap T^{-m} H\left(f\left(\mathbf{z}_j\right), \delta_m+2^{d+2} p^d m^{-3}\right) \right) .
		\end{aligned}
		$$
		Next, applying the polynomial-mixing property of $\mu$, for $m>N_0$,
		$$
		\begin{aligned}
			& \sum_{j \leq r_0^d m^{3 d}} \mu\left(B\left(\mathbf{z}_j, m^{-3}\right) \cap T^{-m} H\left(f\left(\mathbf{z}_j\right), \delta_m+2^{d+2} p^d m^{-3}\right) \right) \\
			\leq & \sum_{j \leq r_0^d m^{3 d}}\left(\mu\left(B\left(\mathbf{z}_j, m^{-3}\right)\right)+\phi(m)\right) \mu\left(H\left(f(\mathbf{z}_j), \delta_m+2^{d+2} p^d m^{-3}\right)\right) \\
			\leq & \sum_{j \leq r_0^d m^{3 d}}\left(\mu\left(B\left(\mathbf{z}_j, m^{-3}\right)\right)+\phi(m)\right) \cdot \mathfrak{c} d 2^{d+1} \delta_m\left(-\log \delta_m\right)^{d-1} \\
			= & \left(\mu(B)+ r_0^d m^{3 d} \phi(m)\right) \cdot \mathfrak{c} d 2^{d+1} \delta_m\left(-\log \delta_m\right)^{d-1} \\
			\leq & \mu(B)\left(1+\mathfrak{c} m^{3 d} \phi(m)\right) \cdot \mathfrak{c} d 2^{d+1} \delta_m\left(-\log \delta_m\right)^{d-1},
		\end{aligned}
		$$
		where the last inequality follows from $\mu(B) \geq \mathfrak{c}^{-1} r_0^d$.
		Together with above estimations, we have
		$$
		\begin{aligned}
			& \mu\left(B \cap D_m \cap D_n\right) \\
			\leq & \left(\mu(B)\left(1+\mathfrak{c} m^{3 d} \phi(m)\right) \cdot \mathfrak{c} d 2^{d+1} \delta_m\left(-\log \delta_m\right)^{d-1}+\tilde{c}_2 r_0^d n^{3 d} \phi(n)^{\frac{1}{2}}\right)   \\ & \cdot \mathfrak{c} d 2^{d+1} \delta_n\left(-\log \delta_n\right)^{d-1}.
		\end{aligned}
		$$
		By the polynomial-mixing property of $\mu$,  there exists a  sufficiently large integer $N_3>N_0$ such that for sufficiently large integers $m$ and $n$ satisfying  $N_3<m < \phi_1^{-1}\left(n^{-4 d}\right)$, 
		$$\mathfrak{c} m^{3 d} \phi(m)<1. $$
		Additionally,  for $n>N_4$, where $ N_4>e^{\frac{\tau N_3-\ln c}{4 d}}$,
		$$
		\tilde{c}_2 r_0^d n^{3 d} \phi(n)^{\frac{1}{2}}<\phi(n)^{\frac{1}{4}}<\phi(m)^{\frac{1}{4}}<\delta_m\left(-\log \delta_m\right)^{d-1}.
		$$
		Therefore, there exist a constant $\tilde{c}_3$, as well as sufficiently large integers $N_3>N_0$ and $ N_4>e^{\frac{\tau N_3-\ln c}{4 d}}$ such that for sufficiently large integers $m$ and $n$ with $N_3<m < \phi_1^{-1}\left(n^{-4 d}\right)$ and $n>N_4$,
		$$
		\mu\left(B \cap D_m \cap D_n\right) \leq \tilde{c}_3 \mu(B) \cdot \delta_m\left(-\log \delta_m\right)^{d-1} \cdot \delta_n\left(-\log \delta_n\right)^{d-1}.
		$$
	\end{proof}
	
	Next, we apply a technique similar to that used in the proof of Lemma \ref{l314} to establish the result in Lemma \ref{l44}.

	\begin{proof}[Proof of Lemma \ref{l44}]
		Let $B(\mathbf{x}, r)$ be a ball with $r<\frac{\min \left\{\delta_1, \delta_2\right\}}{2}$, where $\mathbf{x} \in V_{\delta_1} \cap W_{\delta_2}$ and $\mu(B)>0$.
		Let $N \in \mathbb{N}$, and define $Z_N(x)=\sum_{n=1}^N \chi_{B \cap D_n}(x)$. By Lemma \ref{l45}, for all $n>N_1$, one has
		$$
		\frac{\mu(B)}{2 \mathfrak{c}(d-1)!} \cdot \delta_n\left(-\log \delta_n\right)^{d-1} \leq \mu\left(B \cap D_n\right) \leq \mathfrak{c} d 2^{d+2} \mu(B) \cdot \delta_n\left(-\log \delta_n\right)^{d-1}.
		$$
		Hence,
		$$
		\mathbb{E}\left(Z_N\right)=\sum_{n=1}^{N_1} \mu\left(B \cap D_n\right)+\sum_{n=N_1}^{N} \mu\left(B \cap D_n\right) \geq \frac{\mu(B)}{2 \mathfrak{c}(d-1)!} \sum_{n=N_1}^{N} \delta_n\left(-\log \delta_n\right)^{d-1}.
		$$
		On the other hand,
		$$
		\mathbb{E}\left(Z_N^2\right)=\sum_{m, n=1}^N \mu\left(B \cap D_m \cap D_n\right)=2 \sum_{n=1}^N \sum_{m=1}^{n-1} \mu\left(B \cap D_m \cap D_n\right)+\sum_{n=1}^N \mu\left(B \cap D_n\right).
		$$
		Let $N^*=\max \left(N_1, N_2, N_4\right)$, then 
		$$\sum_{n=1}^N \mu\left(B \cap D_n\right)=\sum_{n=1}^{N^*} \mu\left(B \cap D_n\right) + \sum_{n=N^*}^N \mu\left(B \cap D_n\right).$$
		Additionally, 
		$$
		\begin{aligned}
			&	\sum_{n=1}^N \sum_{m=1}^{n-1} \mu\left(B \cap D_m \cap D_n\right)= 	\sum_{n=1}^{N^*} \sum_{m=1}^{N_3} \mu\left(B \cap D_m \cap D_n\right) + \sum_{n=N^*}^{N} \sum_{m=1}^{N_3} \mu\left(B \cap D_m \cap D_n\right) \\
			&+ \sum_{n=1}^{N^*} \sum_{m=N_3}^{\left\lfloor\phi_1^{-1}\left(n^{-4 d}\right)\right\rfloor} \mu\left(B \cap D_m \cap D_n\right)+ \sum_{n=N^*}^{N} \sum_{m=N_3}^{\left\lfloor\phi_1^{-1}\left(n^{-4 d}\right)\right\rfloor} \mu\left(B \cap D_m \cap D_n\right)  \\
			&+ \sum_{n=1}^{N^*} \sum_{m=\left\lceil\phi_1^{-1}\left(n^{-4 d}\right)\right\rceil}^{n-1} \mu\left(B \cap D_m \cap D_n\right)+ \sum_{n=N^*}^{N} \sum_{m=\left\lceil\phi_1^{-1}\left(n^{-4 d}\right)\right\rceil}^{n-1} \mu\left(B \cap D_m \cap D_n\right).
		\end{aligned}
		$$
		
		By applying Lemmas \ref{l46} and \ref{l48}, there exist $\tilde{c}_4$,$\tilde{c}_5$ and $\tilde{c}_6$ such that
		$$
		\begin{aligned}
			\sum_{n=1}^N \sum_{m=1}^{n-1} \mu\left(B \cap D_m \cap D_n\right) 
			&\leq  \tilde{c}_4 \mu(B) \sum_{n=N^*}^N \sum_{m=\left\lceil\phi_1^{-1}\left(n^{-4 d}\right)\right\rceil}^{n-1}\delta_m\left(-\log \delta_m\right)^{d-1} \cdot \delta_n\left(-\log \delta_n\right)^{d-1} \\
			&+  \tilde{c}_5 \mu(B) \sum_{n=N^*}^N  \delta_n\left(-\log \delta_n\right)^{d-1}+\tilde{c}_6 \mu(B).
		\end{aligned}
		$$
		Hence, there exist $\tilde{c}_7$ and $\tilde{c}_8$,  such that
		$$
		\mathbb{E}\left(Z_N^2\right)  \leq  2 \tilde{c}_4 \mu(B) \left(\sum_{n=N^*}^N \delta_n\left(-\log \delta_n\right)^{d-1}\right)^2 + 
		\tilde{c}_7 \mu(B) \sum_{n=N^*}^N \delta_n\left(-\log \delta_n\right)^{d-1}+ \tilde{c}_{8}\mu(B).
		$$
		By Lemma \ref{l313}, for any $\lambda>0$, we deduce from the previous inequalities that
		$$
		\begin{aligned}
			& \mu\left(Z_N>\lambda \mathbb{E}\left(Z_N\right)\right) \geq(1-\lambda)^2 \frac{\mathbb{E}\left(Z_N\right)^2}{\mathbb{E}\left(Z_N^2\right)} \\
			& \geq(1-\lambda) \frac{\left(\frac{\mu(B)}{2 \mathfrak{c}(d-1)!} \sum_{n=N_1}^N \delta_n\left(-\log \delta_n\right)^{d-1}\right)^2}{2 \tilde{c}_4 \mu(B)\left(\sum_{n=N^*}^N \delta_n\left(-\log \delta_n\right)^{d-1}\right)^2+\tilde{c}_7 \mu(B) \sum_{n=N^*}^N \delta_n\left(-\log \delta_n\right)^{d-1}+\tilde{c}_{8} \mu(B)}.
		\end{aligned}
		$$
		Taking the limit as $N \rightarrow \infty$, we have
		$$
		\begin{aligned}
			\mu\left(B \cap \mathcal{R}^{f \times}\left(\left\{\delta_n\right\}\right)\right) & =\mu\left(\limsup \left(B \cap D_n\right)\right) \geq \mu\left(\limsup \left(Z_N>\lambda \mathbb{E}\left(Z_N\right)\right)\right) \\
			& \geq \limsup \mu\left(Z_N>\lambda \mathbb{E}\left(Z_N\right)\right)\geq (1-\lambda) \frac{\mu(B)}{8 \tilde{c}_4 \mathfrak{c}^2((d-1)!)^2}.
		\end{aligned}
		$$
		Let $\lambda=1 / 2$ and $\alpha_2=\left(16 \tilde{c}_4 \mathfrak{c}^2((d-1)!)^2\right)^{-1}$. Therefore, Lemma \ref{l44} holds, and we reach the conclusion.
	\end{proof}

\end{document}